\documentclass{article}
\usepackage{amsmath,amssymb}  
\usepackage{bm}  
\usepackage{amsthm}
\usepackage{color}

\newtheorem{theorem}{Theorem}
\newtheorem{proposition}[theorem]{Proposition}
\newtheorem{lemma}[theorem]{Lemma}
\newtheorem{cor}[theorem]{Corollary}

\newtheorem{definition}[theorem]{Definition}

\newtheorem{remark}[theorem]{Remark}

\makeatletter
\def\operatorname#1{\mathop{\operator@font #1}\nolimits}%
\makeatother
\newcommand{\R}{\mathbb{R}}
\newcommand{\C}{\mathbb{C}}

\newcommand{\Z}{\mathbb{Z}}

\newcommand{\half}{{\tfrac{1}{2}}}

\newcommand{\Imc}{{\mathcal{I}}m\ }

\newcommand{\Id}{\operatorname{Id}}
\newcommand{\id}{\operatorname{Id}}

\newcommand{\Sp}{\operatorname{Sp}}
\newcommand{\SP}{\operatorname{SP}}

\newcommand{\card}{\operatorname{Card}}

\newcommand{\Gr}{\operatorname{Gr}}

\newcommand{\sign}{\operatorname{Sign}}
\newcommand{\tr}{^{\tau}\!}

\newcommand{\Ker}{\operatorname{Ker}}

\def\adots{\mathinner{\mkern2mu\raise 1pt\hbox{.}\mkern 3mu\raise
4pt\hbox{.}\mkern2mu\raise 8pt\hbox{{.}}}}

\title{Generalized Conley-Zehnder index}
\author{Jean Gutt \\
	\small\hbox{\parbox[t]{1.9in}{\begin{center}
	D{\'e}partement de Math{\'e}matique \\
	Universit{\'e} Libre de Bruxelles \\
	Campus Plaine, C. P. 218 \\
	Boulevard du Triomphe \\
	B-1050 Bruxelles \\
	Belgium\\
	jeangutt\char64ulb.ac.be\end{center}}}
	\hbox{\parbox[t]{.6in}{\begin{center}\rm and \end{center}}}
	\hbox{\parbox[t]{1.9in}{\begin{center}
	Universit\'e de Strasbourg\\
	IRMA\\
	7 rue Ren\'e Descartes\\
	67000 Strasbourg\\
	France\\
	{gutt\char64math.unistra.fr}\end{center}}}
	}
\date{} 
\begin{document} 

\maketitle

\begin{abstract}
The Conley-Zehnder index  associates an integer to any continuous path of symplectic matrices starting from the identity and ending at a matrix which does not admit $1$ as an eigenvalue. We  give new ways to compute this  index. Robbin and Salamon  define a  generalization of the Conley-Zehnder index for any continuous path of symplectic matrices; this generalization is half integer valued. It is based on a Maslov-type index that they define for a continuous path of Lagrangians in a symplectic vector space $(W,\overline{\Omega})$, having chosen a given reference Lagrangian $V$. Paths of symplectic endomorphisms of $(\R^{2n},\Omega_0)$ are viewed as paths of Lagrangians defined by their graphs in $(W=\R^{2n}\oplus \R^{2n},\overline{\Omega}=\Omega_0\oplus -\Omega_0)$ and the reference Lagrangian is the diagonal. Robbin and Salamon give properties of this generalized Conley-Zehnder index and an explicit formula when the path has only regular crossings.
We give here an axiomatic characterization of this generalized Conley-Zehnder index.
We also give an explicit way to compute it for any continuous path of symplectic matrices. 
\end{abstract}




\section{Introduction}
The Conley-Zehnder index  associates an integer to any continuous path $\psi$ defined on the  interval $[0,1]$ with values in the group $\Sp\Bigl(\R^{2n},\Omega_0=\left(\begin{smallmatrix} 0&\Id\\ -\Id&0\end{smallmatrix}\right)\Bigr)$ of $2n\times 2n$ symplectic matrices, starting from the identity and ending at a matrix   which does  not admit $1$ as an eigenvalue.  This index is used in the definition of the grading of Floer homology theories.
If the path $\psi$ were a loop with values in the unitary group, one could define an integer by looking at the degree of the loop in the circle defined by the (complex) determinant -or an integer power of it. The construction  \cite{SalamonZehnder, Salamon, Audin} of the Conley-Zehnder index is based on this idea. One uses
a continuous map $\rho$ from the sympletic group $\Sp(\R^{2n},\Omega_0)$ into $S^1$ and an ``admissible''
extension of  $\psi$ to a path $\widetilde{\psi} : [0,2] \rightarrow \Sp(\R^{2n},\Omega_0)$  in such a way that $\rho^2\circ \widetilde{\psi}:[0,2]\rightarrow S^1$ is a loop. 
The Conley-Zehnder index of $\psi$ is defined as the degree of this loop 
\begin{equation*}
   \mu_{\textrm{CZ}}(\psi) :=\deg (\rho^2\circ \widetilde{\psi}).
\end{equation*}
We recall  this construction in section \ref{CZindex} with the precise definition of the map $\rho$. The value of $\rho(A)$   involves the algebraic multiplicities of the real negative eigenvalues of $A$ and the signature of natural symmetric $2$-forms defined
on the generalised eigenspaces of $A$ for the non real eigenvalues lying on  $S^1$.
We give  alternative ways to compute this index :
\begin{theorem}\label{thm:CZ}
Let $\psi : [0,1] \rightarrow \Sp(\R^{2n},\Omega_0)$ be a  continuous path of matrices  linking the matrix $\Id$ to a matrix which does  not admit $1$ as an eigenvalue.
Let  $\widetilde{\psi} : \left[0,2\right] \rightarrow \Sp(\R^{2n},\Omega_0)$ be an extension such that $\widetilde{\psi}$ coincides with $\psi$ on the interval $\left[0,1\right],$ such that
$\widetilde{\psi}(s)$  does  not admit $1$ as an eigenvalue  for all $s\geqslant 1$ and  such that the path ends either at $\widetilde{\psi}(2) =W^+:= -\Id$ either at
$\widetilde{\psi}(2) = W^- := \operatorname{diag}(2,-1,\ldots,-1,\half,-1,\ldots,-1).$
The Conley-Zehnder index of $\psi$ is equal to the integer given by the degree of the map ${\tilde{\rho}}^2 \circ \tilde{\psi} : \left[0,2\right] \rightarrow \textrm{S}^1:$
	\begin{equation}
		\mu_{\textrm{CZ}}(\psi) = \operatorname{deg}({\tilde{\rho}}^2\circ \widetilde{\psi})
	\end{equation}
for ANY continuous map ${\tilde{\rho}}: \Sp(\R^{2n},\Omega_0)\rightarrow S^1$ with the following properties:
\begin{enumerate}
\item ${\tilde{\rho}}$ coincides with the (complex) determinant $\det_\C$ on $\textrm{U}\left(n\right) = \textrm{O}\left(\R^{2n}\right) \cap \Sp\left(\R^{2n},\Omega_0\right)$; 
\item  ${\tilde{\rho}}(W^-)\in \{\pm 1\}$;
\item	 $\operatorname{deg}\, ({\tilde{\rho}}^2\circ \psi_{2-})= n-1$ \\
for  $\psi_{2-} : t \in [0,1]\mapsto \exp t \pi J_0 \left(\begin{smallmatrix}
		0 & 0 & -\frac{\log 2}{\pi} & 0\\
		0 & \Id_{n-1} & 0 & 0\\
		-\frac{\log 2}{\pi} & 0 & 0 & 0\\
		0 & 0 & 0 & \Id_{n-1}
	\end{smallmatrix}\right)$.
\end{enumerate}
In particular, two alternative ways to compute the Conley-Zehnder index are :
\begin{itemize}
\item   Using the polar decomposition of a matrix,
 \begin{equation}
\mu_{\textrm{CZ}}(\psi) = \operatorname{deg}({{\det}_\C}^2\circ U\circ \widetilde{\psi})
\end{equation}
	where $U:  \Sp(\R^{2n},\Omega_0)\rightarrow \textrm{U}\left(n\right) : A \mapsto AP^{-1}$ with $P$ the unique symmetric positive definite matrix such that $P^2=A\tr A.$
\item Using the normalized determinant of the $\C$-linear part of a symplectic matrix,
 \begin{equation}	\mu_{\textrm{CZ}}(\psi) = \operatorname{deg}(\hat\rho^2\circ \widetilde{\psi})
 \end{equation}
	where 
	$\hat\rho:  \Sp(\R^{2n},\Omega_0)\rightarrow S^1 : A \mapsto 
\hat\rho(A)=\frac{{\det}_{\C} \left(\half(A-J_0AJ_0)\right)}{\left\vert{\det}_{\C} \left(\half(A-J_0AJ_0)\right)\right\vert}$\\
	with $J_0=\left(\begin{smallmatrix} 0&-\id\\\id&0\end{smallmatrix}\right)$ the standard complex structure on $\R^{2n}$.
\end{itemize}
\end{theorem}

In \cite{RobbinSalamon}, Robbin and Salamon define a Maslov-type index for a continuous path $\Lambda$ 
from the interval $[a,b]$ to the space $\mathcal{L}_{(W,\overline{\Omega})} $ of Lagrangian subspaces of a symplectic vector space $(W,\overline{\Omega})$, having chosen a reference Lagrangian $L$.
They give a formula of this index for a path having only regular crossings.
A crossing for $\Lambda$ is a number $t \in [a,b]$ for which $\dim \Lambda_t \cap L \neq 0,$
and a  crossing $t$ is regular if the crossing form $\Gamma (\Lambda , L , t)$ is nondegenerate.
We recall the precise definitions in section \ref{RSLagr}.

Robbin and Salamon define the index of a continuous path of symplectic matrices $\psi :  [0,1] \rightarrow \Sp(\R^{2n},\Omega_0) : t\mapsto \psi_t$  as the  index of the corresponding path of Lagrangians
	in $(W :=\R^{2n}\times\R^{2n} , \overline{\Omega} = -\Omega_0 \times \Omega_0)$ defined by their graphs, 
	\begin{equation*}
	\Lambda= \Gr\psi : [0,1] \rightarrow \mathcal{L}_{(W,\overline{\Omega})} : t \mapsto \Gr\psi_t= \{ (x , \psi_t x) \vert x \in \R^{2n} \}.
	\end{equation*}
The reference Lagrangian is the diagonal $\Delta= \{ (x,x) \vert x \in \R^{2n} \}$.
They prove that this index coincide with the Conley Zehnder index on continuous paths of symplectic matrices which start from the identity and end at a matrix which does not admit $1$ as an eigenvalue.
To be complete, we include this in section \ref{section:RSsympl}. They also prove that this index vanishes on a path
of symplectic matrices with constant dimensional $1$-eigenspace.
Robbin and Salamon  present also another way to associate an index to a continuous path $\psi$ of symplectic matrices.  One chooses a Lagrangian $L$  in $\mathcal{L}_{(\R^{2n},\Omega_0})$ and one considers the 
index  of the path of Lagrangians $ t \mapsto\psi_t L$,  with $L$ as the reference Lagrangian.
We show in section \ref{sect:otherRSindex} that those two indices do not coincide in general.

We use  the normal form of the restriction of a symplectic endomorphism to the generalized eigenspace of eigenvalue $1$ obtained in \cite{Gutt} to construct 
special paths of symplectic endomorphisms with a constant dimension of the eigenspace  of eigenvalue $1$. This leads in section \ref{ssection:char} to a characterization 
of the generalized half-integer valued Conley Zehnder index defined   by Robbin and Salamon :
\begin{theorem}\label{thm:Thechar}
	The Robbin-Salamon index for a continuous path of symplectic matrices is characterized by the following properties:
	\begin{itemize}
		\item\emph{(Homotopy)} it is invariant under homotopies with fixed end points;
		\item\emph{(Catenation)} it is additive under catenation of paths;
		\item\emph{(Zero)} it vanishes on any path $\psi:[a,b]\rightarrow\Sp(\R^{2n},\Omega)$ of matrices such that 
		         $\dim\Ker \bigl(\psi(t)-\Id\bigr) =k   $ is constant on $ [a,b]$;
		\item\emph{(Normalization)} if $S=S\tr \in \R^{2n\times 2n}$ is a symmetric matrix with all eigenvalues of absolute value
			$<2\pi $ and if $\psi(t) = \operatorname{exp}(J_0St)$ for $t\in \left[0,1\right],$ then
			$\mu_{\textrm{RS}}(\psi) = \half \sign S$ where $\sign S$ is the signature of $S$.
	\end{itemize}
\end{theorem}
The same techniques lead in section \ref{ssection:formula} to a new formula for this index :
\begin{theorem}\label{RSexpl}
	Let $\psi : [0,1]\rightarrow\Sp(\R^{2n},\Omega_{0})$ be a path of symplectic matrices.
	Decompose $\psi(0) = \psi^{\star}(0)\diamond \psi^{(1)}(0)$ and 
	$\psi(1) = \psi^{\star}(1)\diamond \psi^{(1)}(1)$
	where $ \psi^{\star}(\cdot)$  does not admit $1$ as eigenvalue and
	$\psi^{(1)}(\cdot)$  is the restriction of $\psi(\cdot)$  to its generalized eigenspace of eigenvalue $1$.
	Consider a continuous extension  $\Psi:[-1,2]\rightarrow \Sp(\R^{2n},\Omega_{0})$ of $\psi$ such that
	\begin{itemize}
		\item  $\Psi(t)=\psi(t)$ for   $t\in [0,1]$;
		\item $\Psi\bigl(-\half\bigr) =\psi^{\star}(0)\diamond
			\left(
				\begin{smallmatrix}
					e^{-1}\Id&0\\
					0& e\Id
				\end{smallmatrix}
			\right)$
			and $\Psi(t)=\psi^{\star}(0) \diamond \phi_0(t)$ where $\phi_0(t)$  has only real positive eigenvalues
			 for $t\in \bigl[-\half,0\bigr]$;
		\item $\Psi\bigl(\frac{3}{2}\bigr)=\psi^{\star}(1)\diamond
			\left(
				\begin{smallmatrix}
					e^{-1}\Id&0\\
					0& e\Id
				\end{smallmatrix}
			\right)$
			and $\Psi(t)=\psi^{\star}(1) \diamond \phi_1(t)$ where $\phi_1(t)$  has only real positive eigenvalues
			 for  $t\in \bigl[1,\frac{3}{2}\bigr]$;
		\item $\Psi(-1)=W^{\pm}$, $\Psi(2)=W^{\pm}$ and  $\Psi(t)$ does not admit $1$ as an eigenvalue for $t\in \bigl[-1,-\half\bigr]$ and for $t\in\bigl[\frac{3}{2},2\bigr]$.
	\end{itemize}
	Then the Robbin Salamon index is given by
	$$
		\mu_{RS}(\psi) =\operatorname{deg}({\tilde{\rho}}^2\circ \Psi)+  \half \sum_{k\ge 1}\sign\Bigl(\hat{Q}_k^{(\psi(0))}\Bigr)
		 -\half \sum_{k\ge 1}\sign\Bigl(\hat{Q}_k^{(\psi(1))}\Bigr)
	$$
	with $\tilde{\rho}$ as in theorem \ref{thm:CZ}, and with 
	\begin{eqnarray*}
		\hat{Q}_k^{A} &:& \Ker\left( (A-\id)^{2k}\right)\times \Ker\left( (A-\id)^{2k}\right)\rightarrow \R\nonumber\\
		&&(v,w)\mapsto\Omega\bigl((A-\id)^{k}v,(A-\id)^{k-1}w\bigr).
	\end {eqnarray*}
\end{theorem} 
In the theorem above, we have used the notation $A\diamond B$ for the symplectic direct sum of two symplectic endomorphisms  with the natural identification of   $\Sp(V',\Omega')\times \Sp(V'',\Omega'' )$ as a subgroup of $\Sp(V'\oplus V'',\Omega'\oplus \Omega'')$. This writes in symplectic basis as $$
A\diamond B:=	\left(\begin{smallmatrix}
		A_1 & 0 & A_2 & 0\\
		0 & B_1 & 0 & B_2\\
		A_3 & 0 & A_4 & 0\\
		0 & B_3 & 0 & B_4
	\end{smallmatrix}\right) \quad \textrm{ for } A=\left(\begin{smallmatrix}
	A_1 & A_2\\
	A_3 & A_4
\end{smallmatrix}\right), \quad  B=\left(\begin{smallmatrix}
	B_1 & B_2\\
	B_3 & B_4
\end{smallmatrix}\right) .
$$

This paper is organized as follows.
We recall the  definition of the Conley-Zehnder index in section \ref{CZindex} and obtain a new way of computing this index in Proposition \ref{compdef} and its corrolaries
(stated above as Theorem \ref{thm:CZ}). In sections \ref{RSLagr} and \ref{section:RSsympl}, we present  known results about the Robbin Salamon index of a  path of Lagrangians and the Robbin Salamon index of a path of symplectic matrices, including the fact that it is a generalization of the Conley-Zehnder index; in section \ref{sect:otherRSindex}, we stress the fact  that another index introduced by Robbin and Salamon does not coincide with this generalization of the
Conley-Zehnder index.
In section \ref{ssection:char}, we give a characterization of the generalization of the Conley-Zehnder index (stated above as Theorem \ref{thm:Thechar}).
Section \ref{ssection:formula} gives a new formula to compute this  index (stated above as Theorem \ref{RSexpl}).

\section*{Acknowledgements}
I thank Fr\'ed\'eric Bourgeois and Alexandru Oancea who introduced me to this subject and  encouraged me to write this text,
and I thank Mihai Damian for the question which has lead to theorem \ref{RSexpl}.
I am grateful to the  Foundation for Scientific Research (FNRS-FRS) for its support.

\section{The Conley-Zehnder index}\label{CZindex}
The Conley-Zehnder index is an application which associates a integer to a continuous path of symplectic matrices starting from the identity and ending at a matrix  in the set $\Sp^\star(\R^{2n},\Omega_0)$ of symplectic matrices which do  not admit $1$
as an eigenvalue. 
\begin{definition}[\cite{SalamonZehnder, Salamon}]
	We consider the set $\SP(n)$ of continuous paths of matrices in $\Sp(\R^{2n},\Omega_0)$ linking the matrix $\Id$ to a matrix in $\Sp^\star(\R^{2n},\Omega_0) :$
	\begin{equation*}
		\SP(n) := \Biggl\{ \psi : \left[0,1\right] \rightarrow \Sp(\R^{2n},\Omega_0)\, \left\vert
		\begin{array}{l}
			 \psi(0)=\Id \textrm{ and }\\
 			1\textrm{ is not an eigenvalue of } \psi(1)
		\end{array}
		\right. \Biggr\}.
	\end{equation*}
\end{definition}
\begin{definition}[\cite{SalamonZehnder, Audin}]
	Let $\rho: \Sp(\R^{2n},\Omega_0)\rightarrow S^1$ be the continuous map defined as follows.
	Given $A\in \Sp(\R^{2n},\Omega)$,
	we consider its eigenvalues $\{ \lambda_i\}$.
	For an eigenvalue $\lambda=e^{i\varphi}\in S^1\setminus\{\pm 1\},$ let  $m^+(\lambda)$ be the number  of positive eigenvalues of the symmetric  non degenerate $2$-form $Q$
	defined on the generalized eigenspace $E_\lambda$ by
	\begin{equation*}
		Q: E_\lambda \times E_\lambda \rightarrow \R :~ (z,z') \mapsto Q(z,z'):=\Imc \Omega_0(z,\overline{z'}).
	\end{equation*}
	Then
	\begin{equation}\label{eqrho}
		{\rho}(A):= (-1)^{\half m^-} \prod_{\lambda\in S^1\setminus\{\pm1\}}\lambda^{\half m^+(\lambda)}
	\end{equation}
	where $m^-$ is the sum of the algebraic multiplicities $m_\lambda=\dim_\C E_\lambda$ of the real negative eigenvalues.
\end{definition}
\begin{proposition}[\cite{SalamonZehnder, Audin}]
	The map $\rho: \Sp(\R^{2n},\Omega_0)\rightarrow S^1$ has the following properties:
	\begin{enumerate}
		\item\label{rho1}[determinant]   $\rho$ coincides with  $\det_\C$ on the unitary subgroup 
			\begin{equation*}
				\rho(A) = {\det}_\C A\textrm{ if } A \in  \Sp(\R^{2n},\Omega_0) \cap \textrm{O}(2n) = \textrm{U}(n);
			\end{equation*}
		\item\label{rho2}[invariance] $\rho$ is invariant under conjugation :
			\begin{equation*}
				\rho(kAk^{-1})=\rho(A)\ \forall k \in \Sp(\R^{2n},\Omega_0);
			\end{equation*}
		\item\label{rho3}[normalisation] $\rho(A)=\pm1$ for matrices which  have no eigenvalue on the unit circle;
		\item\label{rho4}[multiplicativity] $\rho$ behaves multiplicatively with respect to direct sums :  if
			$A=A'\diamond A''$
						with $A'\in \Sp(\R^{2m},\Omega_0)$,  $A''\in \Sp(\R^{2(n-m)},\Omega_0)$
			and $\diamond$  expressing as before the obvious identification of $\Sp(\R^{2m},\Omega_0)\times \Sp(\R^{2(n-m)},\Omega_0 )$ with a subgroup of $\Sp(\R^{2n},\Omega_0)$ then 
			\begin{equation*}
				\rho(A)=\rho(A')\rho(A'').
			\end{equation*}
	\end{enumerate}
\end{proposition}
The  construction  \cite{Salamon, Audin} of the Conley-Zehnder  index 
 is based on the two following facts  
 \begin{itemize}
\item $\Sp^\star(\R^{2n},\Omega_0)$ has two connected components, one containing the  matrix $W^+:=-\id$
and the other containing $$W^-:=\operatorname{diag}(2,-1,\ldots,-1,\half,-1,\ldots,-1);$$
\item any loop in $\Sp^\star(\R^{2n},\Omega_0)$ is contractible in $\Sp(\R^{2n},\Omega_0)$.
 \end{itemize}
Thus any path $\psi:[0,1] \rightarrow \Sp(\R^{2n},\Omega_0)$ in $\SP(n)$ can be extended
to a path $\widetilde{\psi}[0,2] \rightarrow \Sp(\R^{2n},\Omega_0)$ so that 
\begin{itemize}
\item $\widetilde{\psi}(t)={\psi}(t)$ for $t\le 1$;
\item $\widetilde{\psi}(t)$ is in
$\Sp^\star(\R^{2n},\Omega_0)$ for any $t\ge 1$; 
\item $\widetilde{\psi}(2)=W^\pm.$
\end{itemize}
Observe that $\bigl(\rho(\id)\bigr)^2=1$ and $\bigl(\rho(W^\pm)\bigr)^2=1$ so that $\rho^2\circ \widetilde{\psi}:[0,2]\rightarrow S^1$
is a loop in $S^1$ and the contractibility property shows that its degree does not depend on the extension chosen.
\begin{definition}
The Conley-Zehnder index of $\psi$ is defined by:
\begin{equation}
\mu_{\textrm{CZ}} :   \SP(n) \rightarrow \Z :~\psi \mapsto    \mu_{\textrm{CZ}}(\psi) :=\deg (\rho^2\circ \widetilde{\psi})
\end{equation}
for an extension $\widetilde{\psi}$ of $\psi$ as above.
\end{definition}
\begin{proposition}[\cite{Salamon, Audin}]\label{proprietescz}
	The Conley-Zehnder index
		has the following properties:
	\begin{enumerate}
		\item\label{cznaturalite} \emph{(Naturality)} For all path $\phi : \left[0,1\right] \rightarrow \Sp(\R^{2n},\Omega_0)$\ we have
		 	\begin{equation*}
				\mu_{\textrm{CZ}}(\phi\psi\phi^{-1}) = \mu_{\textrm{CZ}}(\psi);
			\end{equation*}
		\item\label{czhomotopy} \emph{(Homotopy)} The Conley-Zehnder index is constant on the components of $ \SP(n);$
		\item \emph{(Zero)} If $\psi(s)$ has no eigenvalue on the unit circle for $s>0$ then
			\begin{equation*}
				\mu_{\textrm{CZ}}(\psi) =0;
			\end{equation*}
		\item \emph{(Product)} If $n'+n''=n,$ , if $\psi'$ is in $\SP(n')$ and $\psi''$in $\SP(n'')$, then 
			\begin{equation*}
				\mu_{\textrm{CZ}}(\psi' \diamond \psi'') =  \mu_{\textrm{CZ}}(\psi') +  \mu_{\textrm{CZ}}(\psi'');
			\end{equation*}
			with the identification of $\Sp(\R^{2n'},\Omega_0)\times \Sp(\R^{2n''},\Omega_0)$ with a subgroup of  $\Sp(\R^{2n},\Omega_0)$;
		\item\label{czlacet} \emph{(Loop)} If $\phi : \left[0,1\right] \rightarrow \Sp(\R^{2n},\Omega_0)$ is a loop with $\phi(0) = \phi(1) = \Id ,$ then
			\begin{equation*}
				\mu_{\textrm{CZ}}(\phi\psi) = \mu_{\textrm{CZ}}(\psi) + 2\mu(\phi)
			\end{equation*}
			where $ \mu(\phi)$ is the Maslov index of the loop $\phi$, i.e.  $\mu(\phi) = \operatorname{deg}(\rho \circ \phi);$
		\item\label{czsignature} \emph{(Signature)} If $S=S\tr \in \R^{2n\times 2n}$ is a symmetric non degenerate matrix with all eigenvalues of absolute value $<2\pi \ (\|S\| < 2\pi)$ and if 
			$\psi(t) = \operatorname{exp}(J_0St)$ for $t\in \left[0,1\right],$ then $\mu_{\textrm{CZ}}(\psi) = \half \sign(S)$ $(\textrm{where } \sign(S) \textrm{ is the signature of }S) .$
		\item\label{determinant} \emph{(Determinant)} $(-1)^{n- \mu_{CZ}(\psi)} = \operatorname{sign}\det\bigr(\Id - \psi(1)\bigl)$
		\item\label{inverse} \emph{(Inverse)} $\mu_{CZ}(\psi^{-1}) = \mu_{CZ}(\psi\tr) = -\mu_{CZ}(\psi)$
	\end{enumerate}
\end{proposition}
\begin{proposition}[\cite{Salamon, Audin}]\label{caract}
	The properties \ref{czhomotopy}, \ref{czlacet} and \ref{czsignature} of homotopy, loop and signature characterize the Conley-Zehnder index.
\end{proposition}
\begin{proof}
	Assume $\mu' : \SP(n) \rightarrow \Z$ is a map satisfying those properties.
	Let $\psi : [0,1] \rightarrow \Sp(\R^{2n},\Omega_0)$ be an element of $\SP(n)$;
	Since $\psi$ is in the same component of $\SP(n)$ as its prolongation $\tilde{\psi} : [0,2] \rightarrow \Sp(\R^{2n},\Omega_0)$ 	we have  $\mu'(\psi) = \mu'(\tilde{\psi}).$

Observe that 	$W^+ = \exp \pi (J_0 S^+)$ with $S^+ = \Id$ and $W^- = \exp \pi (J_0 S^-)$ with
$$
	S^- =
	\left(\begin{smallmatrix}
		0 & 0 & -\frac{\log 2}{\pi} & 0\\
		0 & \Id_{n-1} & 0 & 0\\
		-\frac{\log 2}{\pi} & 0 & 0 & 0\\
		0 & 0 & 0 & \Id_{n-1}
	\end{smallmatrix}\right).
$$
          The catenation of $\tilde{\psi}$ and $\psi_2^-$ (the path $\psi_2$ in the reverse order, i.e followed from end to beginning) when
	$\psi_2 : [0,1] \rightarrow \Sp(\R^{2n},\Omega_0) \ t \mapsto \exp t \pi J_0 S^{\pm}$ is a loop $\phi .$
	Hence $\tilde{\psi}$ is homotopic to the catenation of $\phi$ and $\psi_2 ,$ which  is homotopic to the product $\phi \psi_2 $ (see, for instance, \cite{Gutt}).
	
	Thus we have $\mu'(\psi) = \mu'(\phi \psi_2) .$
	By the loop condition $\mu'(\phi \psi_2) = \mu'(\psi_2) + 2\mu(\phi)$ and by the signature condition $\mu'(\psi_2) = \half \sign (S^{\pm}).$
	Thus 
$$
        \mu'(\psi) = 2\mu(\phi) + \half \sign (S^{\pm}).
$$ 
          Since the same is true for $\mu_{CZ}(\psi)$, this proves uniqueness.
\end{proof}
Remark that we have only used the signature property to know the value of the Conley-Zehnder index on the paths $\psi_{2\pm}:\ t\in [0,1]  \mapsto  \
\exp t \pi J_0 S^{\pm}$.
Hence we  have : 

\begin{proposition}\label{compdef}
Let $\psi\in \SP(n)$ be a  continuous path of matrices in $\Sp(\R^{2n},\Omega_0)$ linking the matrix $\Id$ to a matrix in $\Sp^\star(\R^{2n},\Omega_0)$
and let  $\widetilde{\psi} : \left[0,2\right] \rightarrow \Sp(\R^{2n},\Omega_0)$ be an extension such that $\widetilde{\psi}$ coincides with $\psi$ on the interval $\left[0,1\right],$ such that
$\widetilde{\psi}(s) \in \Sp^{\star}(\R^{2n},\Omega_0)$ for all $s\geqslant 1$ and  such that the path ends either in $\widetilde{\psi}(2) = -\Id=W^+$ either in
$\widetilde{\psi}(2) = W^- := \operatorname{diag}(2,-1,\ldots,-1,\half,-1,\ldots,-1).$
The Conley-Zehnder index of $\psi$ is equal to the integer given by the degree of the map ${\tilde{\rho}}^2 \circ \tilde{\psi} : \left[0,2\right] \rightarrow \textrm{S}^1:$
	\begin{equation}
		\mu_{\textrm{CZ}}(\psi) := \operatorname{deg}({\tilde{\rho}}^2\circ \widetilde{\psi})
	\end{equation}
	for any continuous map ${\tilde{\rho}}: \Sp(\R^{2n},\Omega_0)\rightarrow S^1$ which coincide with the (complex) determinant $\det_\C$ on $\textrm{U}\left(n\right) = \textrm{O}\left(\R^{2n}\right) \cap \Sp\left(\R^{2n},\Omega_0\right)$, such that ${\tilde{\rho}}(W^-)=\pm 1$, and such that
	\[
		\operatorname{deg}\, ({\tilde{\rho}}^2\circ \psi_{2-})= n-1  \quad \textrm{ for }\,  \psi_{2-} : t \in [0,1]\mapsto \exp t \pi J_0 S^-.
	\]		\end{proposition}
\begin{proof}
This is a direct consequence of the fact that the map defined by $ \operatorname{deg}({\tilde{\rho}}^2\circ \widetilde{\psi})$
has the homotopy property, the loop property (since any loop is homotopic to a loop of unitary matrices where $\rho$ and $\det_{\C}$
coincide) and we have added what we need of the signature property to characterize the Conley-Zehnder index. Indeed  $\half \sign S^-=n-1, S^+=\Id_{2n}, \half \sign S^+=n$ and\\
$\exp t \pi J_0S^+=\exp t \pi \left(\begin{smallmatrix}
		0 & -\Id_n\\
		\Id_n&0
                \end{smallmatrix}\right)=\left(\begin{smallmatrix}
		\cos \pi t \Id_n& -\sin \pi t\Id_n\\
		\sin \pi t \Id_n& \cos \pi t \Id_n
                \end{smallmatrix}\right)$ is in $\textrm{U}\left(n\right)$ 
so that
  ${\tilde{\rho}}^2\left(\exp t \pi \left(\begin{smallmatrix}
		0 & -\Id_n\\
		\Id_n&0
                \end{smallmatrix}\right)\right)= e^{2\pi int}$ and
$\operatorname{deg}({\tilde{\rho}}^2\circ \psi_{2+})=n$.
\end{proof}
\begin{cor}
The Conley-Zehnder index of a path $\psi\in \SP(n)$ is given by
\begin{equation}\label{defdetunit}
		\mu_{\textrm{CZ}}(\psi) := \operatorname{deg}({{\det}_\C}^2\circ U\circ \widetilde{\psi})
\end{equation}
	where $U:  \Sp(\R^{2n},\Omega_0)\rightarrow \textrm{U}\left(n\right)$ is the projection defined by the polar
	decomposition $U(A)=AP^{-1}$ with $P$ the unique symmetric positive definite matrix such that $P^2=A\tr A.$
\end{cor}
\begin{proof}
	The map $\tilde{\rho}:={{\det}_\C}\circ U$ satisfies all the  properties stated in proposition \ref{compdef};
	it is indeed continuous, coincides obviously with $\det_{\C}$ on $\textrm{U}\left(n\right)$ and we have that 
	$\exp t \pi J_{0}
	\left(
	  	\begin{smallmatrix}
			0 & -\frac{\log 2}{\pi}\\
			-\frac{\log 2}{\pi}&0
		\end{smallmatrix}
	\right) =
	\left(
		\begin{smallmatrix}
			2^t&  0\\
			0& 2^{-t} 
		\end{smallmatrix}
	\right)$
	is a positive symmetic matrix so that
	$ U(\exp t \pi J_0 S^-)=
	\left(
		\begin{smallmatrix}
			1&0&0&0\\
			0&\cos \pi t \Id_{\tiny{n-1}}& 0&-\sin \pi t\Id_{\tiny{n-1}}\\
			0&0&1&0\\
			0&\sin \pi t \Id_{\tiny{n-1}}& 0& \cos \pi t \Id_{\tiny{n-1}}\\
		\end{smallmatrix}
	\right)$;\\
	hence ${\det}_{\C}^2\circ U(\exp t \pi J_0 S^-)=e^{2\pi i(n-1)t}$
	and $\operatorname{deg}({{\det}_\C}^2\circ U\circ \psi_{2-})=n-1$.
\end{proof}
		  
Formula \eqref{defdetunit} is the definition of the Conley-Zehnder index used in \cite{DeGosson, HWZ}.
Another formula is obtained using the parametrization of the symplectic group introduced in \cite{refs:RobRaw}:
\begin{cor}
The Conley-Zehnder index of a path $\psi\in \SP(n)$ is given by
\begin{equation}\label{newdef}
		\mu_{\textrm{CZ}}(\psi) := \operatorname{deg}(\hat\rho^2\circ \widetilde{\psi})
\end{equation}
	where $\hat\rho:  \Sp(\R^{2n},\Omega_0)\rightarrow S^1$ is the normalized complex determinant
	of the $\C$-linear part of the matrix:
	\begin{equation}
\hat\rho(A)=\frac{{\det}_{\C} \left(\half(A-J_0AJ_0)\right)}{\left\vert{\det}_{\C} \left(\half(A-J_0AJ_0)\right)\right\vert}.	
	\end{equation}
\end{cor}
\begin{proof}
	Remark that for any $A \in \Sp(\R^{2n},\Omega_0)$ the element $C_A:=\half(A-J_0AJ_0)$, which clearly defines
	a complex linear endomorphism of $\C^{n}$ since it commutes with $J_0$, is always invertible.
	Indeed for any non-zero $v \in V$ 
	\[
		 4\Omega_0 (C_Av, J_0C_Av) \, = \, 2\Omega_0 (v,J_0v) + \Omega_0 (Av, J_0Av) +
		\Omega_0 (AJ_0v, J_0AJ_0v) > 0.
	\]
	If $A\in \textrm{U}\left(n\right)$, then $C_A=A$ so that $\hat\rho(A)={\det}_{\C}(A)$ hence
	$\hat\rho$ is a continuous map which coincide with $\det_{\C}$ on $\textrm{U}\left(n\right)$.
	Furthermore\\
	$\half
	\left(\left(
		\begin{smallmatrix}
			2^t&  0\\
			0& 2^{-t} 
		\end{smallmatrix}
	\right)
	- J_0
	\left(
		\begin{smallmatrix}
			2^t&  0\\
			0& 2^{-t} 
		\end{smallmatrix}
	\right)
	J_0\right)=\half
	\left(
		\begin{smallmatrix}
			2^t+2^{-t}&  0\\
			0& 2^t + 2^{-t} 
		\end{smallmatrix}
	\right)$
	hence its complex determinant is equal to $\half(2^t+2^{-t})$ and its normalized complex determinant is equal to $1$
	so that $\hat\rho(\exp t \pi J_0 S^-)=e^{\pi i (n-1)t}$ and $\operatorname{deg}(\hat\rho^2\circ \psi_{2-})=n-1.$
\end{proof}

\section{The Robbin-Salamon index for a path of Lagrangians}\label{RSLagr}

A \emph{Lagrangian} in  a symplectic vector space $(V,\Omega)$  of dimension $2n$ is a subspace $L$ of $V$ of dimension $n$ such that $\left. \Omega \right\vert_{L \times L} = 0 .$
Given any Lagrangian $L$ in $V ,$ there exists a Lagrangian $M$ (not unique!) such that $L \oplus M = V .$
With the choice of such a supplementary $M$ any Lagrangian $L'$ in a neighborhood of $L$ (any Lagrangian supplementary to $M$) can be identified to a linear map $\alpha : L \rightarrow M$ through
$L' = \{ v + \alpha (v) \,\vert \, v \in L \} ,$ with $\alpha$ such that $\Omega\bigl(\alpha(v),w\bigr) + \Omega\bigl(v,\alpha(w)\bigr) = 0 \ \forall v,w \in L$.
Hence it can be identified to a symmetric bilinear form $\underline{\alpha} : L \times L \rightarrow \R : (v,v') \mapsto \Omega\bigl(v,\alpha(v')\bigr) .$
In particular the tangent space at a point $L$ to the space $\mathcal{L}_{(V,\Omega)}$ of Lagrangians in $(V,\Omega)$ can be identified to the space of symmetric bilinear forms on $L .$

If $\Lambda : [a,b] \rightarrow \mathcal{L}_n:=\mathcal{L}_{(\R^{2n},\Omega_0)} : t \mapsto \Lambda_t$ is a smooth curve of Lagrangian subspaces in $(\R^{2n},\Omega_0)$ , we define $Q(\Lambda_{t_0}, \dot{\Lambda}_{t_0})$
to be the symmetric bilinear form on $\Lambda_{t_0}$ defined by
\begin{equation}
	Q(\Lambda_{t_0}, \dot{\Lambda}_{t_0})(v,v') = \left.\frac{d}{dt} \underline{\alpha}_t (v,v')\right\vert_{t_0} = \left.\frac{d}{dt} \Omega \bigl(v,\alpha_t(v')\bigr) \right\vert_{t_0}
\end{equation}
where $\alpha_t : \Lambda_{t_0} \rightarrow M$ is the map corresponding to $\Lambda_t$ for a decomposition $\R^{2n} = \Lambda_{t_0} \oplus M$ with $M$ Lagrangian.
Then \cite{RobbinSalamon} :
\begin{itemize}
\item	the symmetric bilinear form $Q(\Lambda_{t_0}, \dot{\Lambda}_{t_0}) : \Lambda_{t_0} \times \Lambda_{t_0} \rightarrow \R$ is independent of the choice of the supplementary Lagrangian $M$ to $\Lambda_{t_0} ;$
\item
if $\psi \in \Sp(\R^{2n},\Omega_0)$ then
	\begin{equation}\label{lem:naturalitedeQ}
		Q(\psi \Lambda_{t_{0}}, \psi \dot{\Lambda}_{t_{0}})(\psi v, \psi v') = Q(\Lambda_{t_{0}}, \dot{\Lambda}_{t_{0}})(v,v') \quad \forall v,v' \in \Lambda_{t_{0}} .
	\end{equation}
\end{itemize}

Let us choose and fix a Lagrangian $L$ in $(\R^{2n},\Omega_0)$. Consider a smooth path of Lagrangians $\Lambda : [a,b] \rightarrow \mathcal{L}_n$.
A \emph{crossing} for $\Lambda$ is a number $t \in [a,b]$ for which $\dim \Lambda_t \cap L \neq 0 .$
At each crossing time $t\in [a,b]$ one defines the \emph{crossing form}
\begin{equation}
	\Gamma (\Lambda , L , t) = \left. Q \bigl( \Lambda_t , \dot{\Lambda}_t \bigr) \right\vert_{\Lambda_t \cap L} .
\end{equation}
A crossing $t$ is called \emph{regular} if the crossing form $\Gamma (\Lambda , L , t)$ is nondegenerate.
In that case $\Lambda_s \cap L = \{0\}$ for $s\neq t$ in a neighborhood of $t.$
\begin{definition}[\cite{RobbinSalamon}]
	For a curve $\Lambda : [a,b] \rightarrow \mathcal{L}_n$ with only regular crossings the \emph{Robbin-Salamon index} is defined as
	\begin{equation}\label{Maslov}
		\mu_{RS}(\Lambda , L) = \half \sign \Gamma(\Lambda, L, a) + \sum_{{\stackrel{a<t<b}{\mbox{\tiny $t$ crossing}}}} \sign \Gamma(\Lambda, L, t) + \half \sign \Gamma(\Lambda, L, b) .
	\end{equation}
\end{definition}
Robbin and Salamon show (Lemmas $2.1$ and $2.2$  in \cite{RobbinSalamon}) that two paths with only regular crossings which are homotopic with fixed endpoints have the same Robbin-Salamon index
and that every continuous path of Lagrangians is homotopic with fixed endpoints to one having only regular crossings.
These two properties allow to define the Robbin-Salamon index for every continuous path of Lagrangians and this index is clearly invariant under homotopies with fixed endpoints.
It depends on the choice of the reference Lagrangian $L$.
Robbin and Salamon show (\cite{RobbinSalamon}, Theorem $2.3$):
\begin{theorem}[\cite{RobbinSalamon}]\label{proplag}
	The index $\mu_{RS}$ has the following properties:
	\begin{enumerate}
		\item(Naturality) For $\psi \in \Sp(\R^{2n},\Omega) ~~$
			$\mu_{RS}(\psi \Lambda,\psi L ) = \mu_{RS}(\Lambda, L )$.
		\item(Catenation) For $a < c < b,~
			\mu_{RS}(\Lambda, L) = \mu_{RS}(\Lambda_{\vert_{[a,c]}}, L ) + \mu_{RS}(\Lambda_{\vert_{[c,b]}}, L )$.
		\item(Product) If $n'+n''= n$, identify $\mathcal{L}_{n'}\times \mathcal{L}_{n''}$ as a submanifold of $\mathcal{L}_{n}$ in the obvious way.
			Then $\mu_{RS}(\Lambda'\oplus\Lambda'' , L'\oplus L'') = \mu_{RS}(\Lambda', L') + \mu_{RS}(\Lambda'', L'').$
		\item(Localization) If $L = R^n \times \{0\}$ and $\Lambda(t) = \Gr(A(t))$ where $A(t)$ is a path of symmetric matrices, then the  index
			of  $\Lambda$ is given by\\ $\mu_{RS}(\Lambda, L ) = \half \sign A(b) - \half \sign A(a)$.
		\item(Homotopy) Two paths $\Lambda_0, \Lambda_1 : [a, b] \rightarrow \mathcal{L}_{n}$ with $\Lambda_0(a) =\Lambda_1(a)$ and $\Lambda_0(b) = \Lambda_1(b)$ are homotopic with fixed endpoints if and only if they
			have the same index.
		\item(Zero) Every path $\Lambda : [a, b] \rightarrow \Sigma_k(V )$, with $\Sigma_k(V )=\{\,M\in \mathcal{L}_{n}\,\vert\, \dim M\cap L=k\,\}$, has  index $\mu_{RS}(\Lambda, L ) = 0$.
	\end{enumerate}
\end{theorem}

\section{The Robbin-Salamon index for a path of symplectic matrices}\label{section:RSsympl}
\subsection{Generalized Conley-Zehnder index}
Consider the symplectic vector space $(\R^{2n}\times\R^{2n} , \overline{\Omega} = -\Omega_0 \times \Omega_0).$
Given any linear map $\psi: \R^{2n} \rightarrow \R^{2n} ,$ its graph
\begin{equation*}
	\Gr\psi = \{ (x , \psi x) \vert x \in \R^{2n} \}
\end{equation*}
is a $2n$-dimensional subspace of $\R^{2n}\times\R^{2n}$ which is Lagrangian
if and only if $\psi$ is symplectic $\bigl(\psi \in \Sp(\R^{2n},\Omega_0)\bigr) .$

A particular Lagrangian is given by the diagonal
\begin{equation}
	\Delta = \Gr\Id = \{ (x,x) \vert x \in \R^{2n} \} .
\end{equation}
Remark that $\Gr (-\psi)$ is a Lagrangian subspace which is always supplementary to $\Gr\psi$ for $\psi \in \Sp(\R^{2n},\Omega_0) .$
In fact $\Gr \phi$ and $\Gr\psi$ are supplementary if and only if $\phi-\psi$ is invertible.
\begin{definition}[\cite{RobbinSalamon}]\label{RS}
	The \emph{Robbin-Salamon index} of a continuous path of symplectic matrices $\psi :  [0,1] \rightarrow \Sp(\R^{2n},\Omega_0) : t\mapsto \psi_t$ is defined as the Robbin-Salamon index of the path of Lagrangians
	in $(\R^{2n}\times\R^{2n} , \overline{\Omega})$, 
	\begin{equation*}
		\Lambda= \Gr\psi : [0,1] \rightarrow \mathcal{L}_{2n} : t \mapsto \Gr\psi_t
	\end{equation*}
	when the fixed Lagrangian is the diagonal $\Delta$:
	\begin{equation}\label{hungry}
		\mu_{RS}(\psi):=\mu_{RS}( \Gr\psi  , \Delta).
	\end{equation}
\end{definition}
Note that this index is defined for any continuous path of symplectic matrices  but can have half integer values.

A crossing for a smooth path $\Gr\psi$ is a number $t\in [0,1]$ for which $1$ is an eigenvalue of $\psi_t$ and 
\begin{equation*}
	\Gr\psi_t\cap \Delta=\{\, (x,x)\,\vert\, \psi_tx=x\,\}
\end{equation*}
is in bijection with $\Ker(\psi_t-\Id).$

The properties of homotopy, catenation and product of theorem \ref{proplag} imply that \cite{RobbinSalamon}
\begin{itemize}
	\item $\mu_{RS}$ is invariant under homotopies with fixed endpoints,
	\item $\mu_{RS}$ is additive under catenation of paths and
	\item $\mu_{RS}$ has the product property $\mu_{RS}(\psi'\diamond\psi'') = \mu_{RS}(\psi')+\mu_{RS}(\psi'')$ as in proposition \ref{proprietescz}.
\end{itemize}
The zero property of the Robbin-Salamon index of a path of Lagrangians becomes:
\begin{proposition}
If $\psi:[a,b]\rightarrow\Sp(\R^{2n},\Omega)$ is a path of matrices such that $\dim\Ker (\psi(t)-\Id) =k$ for all $t\in [a,b]$
then $\mu_{RS}(\psi)=0$.
\end{proposition}
Indeed, $\Gr\psi_t\cap\Delta = \{ v\in\R^{2n} \vert \psi_tv=v\}$ so
$\dim (\Gr\psi_t \cap \Delta) = k$ if and only if $\dim\Ker (\psi(t)-\Id) =k$.
\begin{proposition}[Naturality]\label{prop:invariance}
	Consider two continuous paths of symplectic matrices $\psi , \phi :  [0,1] \rightarrow \Sp(\R^{2n},\Omega_0)$ and define $\psi' = \phi\psi\phi^{-1}$.
	Then
	\begin{equation*}
		\mu_{RS}(\psi')=\mu_{RS}(\psi)
	\end{equation*}
\end{proposition}
\begin{proof}
	One has
	\begin{eqnarray*}
		\Lambda'_t := \Gr \psi'_t &=& \{ (x,\phi_t\psi_t\phi_t^{-1}x) \,\vert \,x\in \R^{2n}\}\\
		&=& \{ (\phi_ty,\phi_t\psi_ty) \,\vert\, y\in \R^{2n}\}\\
		&=& (\phi_t\times\phi_t) \Gr \psi_t \\
		&=& (\phi_t\times\phi_t)\Lambda_t
	\end{eqnarray*}
	and $(\phi_t\times\phi_t)\Delta = \Delta$.
	Furthermore $(\phi_t\times\phi_t) \in \Sp(\R^{2n}\times\R^{2n} , \overline{\Omega}).$\\	
	Hence $t \in [0,1]$ is a crossing for the path of Lagrangians $\Lambda' = \Gr\psi'$ if and only if
	$\dim \Gr\psi'_t \cap \Delta \neq 0$ if and only if $\dim (\phi_t \times \phi_t)(\Gr\psi_t \cap \Delta) \neq 0$
	if and only if $t$ is a crossing for the path of Lagrangian $\Lambda = \Gr\psi$.
	
	By homotopy with fixed endpoints, we can assume that $\Lambda$ has only regular crossings and
	$\phi$ is locally constant around each crossing $t$ so that
	\begin{equation*}
		\frac{d}{dt}({\phi\psi\phi^{-1}})(t) = \phi_t\dot{\psi}_t\phi^{-1}_t .
	\end{equation*}
	Then at each crossing
	\begin{eqnarray*}
		\Gamma(\Gr\psi',\Delta,t) &=& Q(\Lambda'_t , \dot{\Lambda'}_t)\vert_{\Gr\psi'_t\cap\Delta}\\
		&=& Q((\phi_t \times \phi_t)\Lambda_t , (\phi_t \times \phi_t)\dot{\Lambda}_t)\vert_{(\phi_t \times \phi_t)\Gr\psi_t\cap\Delta}\\
		&=& Q(\Lambda_t , \dot{\Lambda}_t)\vert_{\Gr\psi_t\cap\Delta} \circ (\phi^{-1}_t \times \phi^{-1}_t)\otimes(\phi^{-1}_t \times \phi^{-1}_t)
	\end{eqnarray*}
	in view of  \eqref{lem:naturalitedeQ}, so that
	\begin{equation*}
		\sign \Gamma(\Gr\psi',\Delta,t) = \sign \Gamma(\Gr\psi,\Delta,t).
	\end{equation*}
\end{proof}
\begin{definition}
	For any smooth path $\psi$ of symplectic matrices, define a path of symmetric matrices $S$ through
	\begin{equation*}
		\dot{\psi}_t = J_0S_t\psi_t.
	\end{equation*}
	This is  indeed  possible since $\psi_t\in\Sp(\R^{2n},\Omega_0)\, \forall t$, thus $\psi_t ^{-1}\dot{\psi}_t$ is in the Lie algebra $sp(\R^{2n},\Omega_0)$
	and every element of this Lie algebra may be written in the form $J_0S$ with $S$ symmetric.
\end{definition}
The symmetric bilinear form $Q\bigl(\Gr\psi,\frac{d}{dt}{\Gr\psi}\bigr)$ is given as follows.
For any $t_0\in[0,1],$ write $\R^{2n}\times\R^{2n}=\Gr\psi_{t_0}\oplus \Gr(-\psi_{t_0}).$
The linear map $\alpha_t : \Gr\psi_{t_0}\rightarrow \Gr(-\psi_{t_0})$ corresponding to $\Gr\psi_t$ is obtained from:
\begin{equation*}
	(x,\psi_tx)=(y,\psi_{t_0}y) + \alpha_t(y,\psi_{t_0}y)=(y,\psi_{t_0}y)+(\widetilde{\alpha}_t y,-\psi_{t_0}\widetilde{\alpha}_t y)
\end{equation*}
if and only if $(\Id+\widetilde{\alpha}_t )y=x$ and $\psi_{t_0}(\Id-\widetilde{\alpha}_t )y=\psi_tx$, hence $\psi_{t_0}^{-1}\psi_t(\Id +\widetilde{\alpha}_t)=\Id - \widetilde{\alpha}_t$ and
\begin{equation*}
	\widetilde{\alpha}_t=(\Id +\psi_{t_0}^{-1}\psi_t)^{-1}(\Id-\psi_{t_0}^{-1}\psi_t)\quad;\qquad \left. \frac{d}{dt}\widetilde{\alpha}_t \right\vert_{t_0}=-\half\psi_{t_0}^{-1}\dot\psi_{t_0}.
\end{equation*}
Thus
\begin{eqnarray*}
	\lefteqn{Q \Bigl( \Gr\psi_{t_0},\frac{d}{dt}\Gr\psi_{t_0} \Bigr) \bigl( ( v,\psi_{t_0}v),(v',\psi_{t_0}v' ) \bigr) }\\
	& = & \left. \frac{d}{dt} \overline{\Omega} \bigl( (v,\psi_{t_0}v), \alpha_t(v',\psi_{t_0}v') \bigr) \right\vert_{t_0}\\
	& = & \left.\frac{d}{dt} \overline{\Omega} \bigl( (v,\psi_{t_0}v), (\widetilde{\alpha}_tv',-\psi_{t_0} \widetilde{\alpha}_tv' ) \bigr) \right\vert_{t_0}\\
	& = & -2\Omega_0 \Bigl( v,\left.\frac{d}{dt} \widetilde{\alpha_t} \right\vert_{t_0}v' \Bigr)\\
	& = & \Omega_0(v,\psi_{t_0}^{-1}\dot\psi_{t_0}v')\\
	& = & \Omega_0(\psi_{t_0}v,J_0S_{t_0}\psi_{t_0}v').
\end{eqnarray*}
Hence the restriction of $Q$ to $\Ker(\psi_{t_0}-\Id)$ is given by
\begin{equation*}
	Q \Bigl( \Gr\psi_{t_0},\frac{d}{dt}\Gr\psi_{t_0} \Bigr) \bigl( ( v,\psi_{t_0}v),(v',\psi_{t_0}v' ) \bigr) = v^\tau S_{t_0}v' \quad \forall v,v' \in \Ker(\psi_{t_0}-\Id).
\end{equation*}
A crossing $t_0\in [0,1]$ is thus regular for the smooth path $\Gr\psi$ if and only if the restriction of $S_{t_0}$ to $\Ker(\psi_{t_0}-\Id)$ is nondegenerate.
\begin{definition}[\cite{RobbinSalamon}]\label{cross}
	Let $\psi :  [0,1] \rightarrow \Sp(\R^{2n},\Omega_0) : t\mapsto \psi_t$ be a smooth path of symplectic matrices. Write $\dot\psi_t=J_0S_t \psi_t$ with $t\mapsto S_t$ a path of symmetric matrices.
	A number $t \in \left[0,1\right]$ is called a \emph{crossing} if $\det (\psi_t-\Id) = 0.$
	For $t\in [0,1]$,  the \emph{crossing form} $\Gamma(\psi,t)$ is defined as the quadratic form which is the restriction of $S_t$ to $\Ker (\psi_t-\Id).$
A crossing $t_0$ is called \emph{regular} if the crossing form $\Gamma(\psi,t_0)$ is nondegenerate.		
\end{definition}
\begin{proposition}[\cite{RobbinSalamon}]\label{murs}
	For a smooth path $\psi:  [0,1] \rightarrow \Sp(\R^{2n},\Omega_0) : t\mapsto \psi_t$ having only regular crossings, the Robbin-Salamon index introduced in definition \ref{RS} is given by
	\begin{equation}\label{muCZRS}
		\mu_{RS}(\psi) = \half \sign \Gamma(\psi,0) + \sum_{{\stackrel{t \textrm{ crossing,}}{\mbox{\tiny{$t\in ]0,1[$}}}}} \sign \Gamma(\psi,t)+\half \sign \Gamma(\psi,1).
	\end{equation}
\end{proposition}
\begin{proposition}[\cite{RobbinSalamon}]
	Let $\psi : [0,1] \rightarrow \Sp(\R^{2n},\Omega_0)$ be a continuous path of symplectic matrices such
	that $\psi(0)=\Id$ and such that $1$ is not an eigenvalue of $\psi(1)$ (i.e. $\psi \in \SP(n)$).
	The Robbin-Salamon index of $\psi$ defined by \eqref{hungry} coincides with the Conley-Zehnder index of $\psi$ 
	In particular, for a smooth path  $\psi \in \SP(n)$ having only regular crossings, the Conley-Zehnder index is given by
	\begin{eqnarray}
		\mu_{CZ}(\psi)&=&  \half \sign \Gamma(\psi,0)+ \sum_{{\stackrel{t \textrm{ crossing,}}{\mbox{\tiny{$t\in ]0,1[$}}}}} \sign \Gamma(\psi,t) \nonumber\\
		&=&\half\sign(S_0)+ \sum_{{\stackrel{t \textrm{ crossing,}}{\mbox{\tiny{$t\in ]0,1[$}}}}} \sign \Gamma(\psi,t) 
	\end{eqnarray}
	with  $S_0=-J_0\dot\psi_0.$
\end{proposition}
\begin{proof}
	Since the Robbin-Salamon index for paths of Lagrangians is invariant under homotopies with fixed end points, the Robbin-Salamon index  for paths of symplectic matrices is also invariant under homotopies  with fixed endpoints.
	
	Its restriction to $\SP(n)$ is actually invariant under homotopies of paths in $\SP(n)$ since for any path in $\SP(n)$, the starting point $\psi_0=\Id$  is fixed and the endpoint $ \psi_1$ can only move
	in a connected component of $\Sp^*(\R^{2n},\Omega_0)$ where no matrix has $1$ as an eigenvalue.

	To show that this index coincides with the Conley-Zehnder index, it is enough, in view of proposition \ref{caract}, to show that it satisfies the loop and signature properties.

	Let us prove the signature property. Let $\psi_t=\exp(tJ_0S)$ with $S$ a symmetric nondegenerate matrix with all eigenvalues of absolute value $ <2\pi$, so that 
	$\Ker(\exp(tJ_0S)-\Id)=\{0\}$ for all $t\in ]0,1] .$ Hence the only crossing is at $t=0$, where $\psi_0=\Id$ and $\dot\psi_t= J_0 S \psi_t$ so that $S_t=S$ for all $t$ and 
	\begin{equation*}
		\mu_{CZ}(\psi)=\half\sign S_0=\half\sign S.
	\end{equation*}

	To prove the loop property, note that $\mu_{RS}$ is additive for catenation and invariant under homotopies with fixed endpoints.
	The path  $(\phi\psi)$ is  homotopic to the catenation of $\phi$ and $\psi$;
	it is thus enough to show that the Robbin-Salamon index of a loop is equal to $2\deg (\rho\circ \phi).$
	Since two loops $\phi$ and $\phi'$ are homotopic if and only if $\deg(\rho\circ\phi)=\deg(\rho\circ\phi'),$ it is enough to consider the loops $\phi_n$ defined by
	\begin{equation*}
		\phi_n(t):=
		\left(\begin{matrix}
			\cos 2\pi nt & -\sin 2\pi nt \\
			\sin 2\pi nt & \cos 2\pi nt\end{matrix}\right)\diamond \left(\begin{matrix}
			a(t)\Id &0\\
			0& a(t)^{-1}\Id\end{matrix}\right)
			\end{equation*}
	with $a: [0,1] \rightarrow \R^+$ a smooth curve with $a(0)=a(1)=1$ and $a(t)\neq 1$ for $t\in ]0,1[.$
	Since $\rho\bigl(\phi_n(t)\bigr)=e^{2\pi int}$, we have $\deg(\phi_n)=n$.\\
	The crossings of $\phi_n$ arise at $t=\frac{m}{n}$ with $m$ an integer between $0$ and $n.$ At such a crossing, $\Ker\bigl(\phi_n(t)\bigr)$ is $\R^2$ for $0<t<1$ and is $\R^{2n}$ for $t=0$ and $t=1.$ We have
	\begin{equation*}
		\dot\phi_n(t)=
		\left(\left(\begin{smallmatrix}
			0 & -2\pi n \\
			2\pi n&0\end{smallmatrix}\right)\diamond \left(\begin{smallmatrix}
			 \frac{\dot a(t)}{a(t)}\Id & 0\\
			0& -\frac{\dot a(t)}{a(t)}\Id
		 \end{smallmatrix}\right)\right)
		 \phi_n(t)
	\end{equation*}
	so that, extending $\diamond$ to symmetric matrices in the obvious way,
	\begin{equation*}
		S(t)=
		\left(\begin{smallmatrix}
			2\pi n &0\\
			0&2\pi n\end{smallmatrix}\right)\diamond \left(\begin{smallmatrix}
			0&- \frac{\dot a(t)}{a(t)}\Id \\
			 -\frac{\dot a(t)}{a(t)}\Id&0
		\end{smallmatrix}\right).
	\end{equation*}
	Thus  $\sign \Gamma(\phi_n,t)=2$ for all crossings $t=\frac{m}{n} ,\ 0 \le m \le n$.
	From equation \eqref{muCZRS} we get
	\begin{eqnarray*}
		\mu_{RS}(\phi_n)&=&\half\sign \Gamma(\phi_n,0)+\sum_{0<m<n} \sign \Gamma\bigl(\phi_n,\tfrac{m}{n}\bigr)+\half\sign \Gamma(\phi_n,1)\\
		&=&1+2(n-1)+1=2n=2\deg(\rho\circ \phi_n)
	\end{eqnarray*}
	and the loop property is proved.
	Thus the Robbin-Salamon index for paths in $\SP(n)$ coincides with the Conley-Zehnder index.

	The formula for the Conley-Zehnder index of a path $\psi \in \SP(n)$ having only regular crossings, follows then from \eqref{muCZRS}.
	Indeed, we have $\Ker(\psi_1-\Id)=\{0\}$, while $\Ker(\psi_0-\Id)=\R^{2n}$ and $\Gamma(\psi,0)=S_0.$
\end{proof}
\subsection{Another  index defined by Robbin and Salamon}\label{sect:otherRSindex}
\begin{definition}
A \emph{symplectic shear} is a path of symplectic matrices of the form $\psi_t =
	\left(\begin{smallmatrix}
		\Id & B(t) \\
		0 & \Id
	\end{smallmatrix} \right)$
	with $B(t)$ symmetric.
\end{definition}
\begin{proposition}\label{RSforshears}
	The Robbin-Salamon index  of a symplectic shear $\psi_t =
	\left(\begin{smallmatrix}
		\Id & B(t) \\
		0 & \Id
	\end{smallmatrix} \right),$
	with $B(t)$ symmetric, is equal to 
	$$
		\mu_{\textrm{RS}}(\psi)=\half \sign B(0) - \half \sign B(1).
	$$
\end{proposition}
\begin{proof}
 We  write $B(t)= A(t)^\tau D(t)A(t)$ with $A(t)\in O(\R^{n})$  and  $D(t)$ a diagonal matrix.
          The matrix $\phi_t=
          \left(\begin{smallmatrix}
		A(t)^\tau & 0 \\
		0 & A(t)
	\end{smallmatrix} \right)$
	is in $\Sp(\R^{2n},\Omega_0)$ and 
         \begin{equation*}
		\psi'_t:=\phi_t\psi_t\phi_t^{-1}=
		\left(\begin{smallmatrix}
			\Id & D(t) \\
			0 & \Id
		\end{smallmatrix} \right).
        \end{equation*}
	By proposition \ref{prop:invariance} $\mu_{\textrm{RS}}(\psi)=\mu_{\textrm{RS}}(\psi');$ by the product property it is enough to show that $\mu_{\textrm{RS}}(\psi)=\half \sign d(0)-\half \sign d(1)$ for the path 
	$$
		\psi: [0,1]\rightarrow \Sp(\R^2,\Omega_0) : t \mapsto \psi_t=
		\left(\begin{smallmatrix}
			1 & d(t) \\
			0 & 1
		\end{smallmatrix} \right).
	$$
	Since $\mu_{\textrm{RS}}$ is invariant under homotopies with fixed end points, we may assume $\psi_t=
	\left(\begin{smallmatrix}
		a(t)& d(t) \\
		c(t) &a(t)^{-1}\bigl(1+ d(t)c(t)\bigr)
	\end{smallmatrix} \right)$
	with $a$ and $c$  smooth functions such that 
	$a(0)=1$, $a(1)=1, ~\dot a(0)\neq 0,\, \dot a(1)\neq 0$ and $ a(t)>1$ for $0<t<1;$
	$c(0)=c(1)=0$, $c(t)d(t)\ge 0~ \forall t$ and $\dot c(t)\neq 0 $ (resp.$=0$) when $d(t)\neq 0 $ (resp.$=0$) for $t=0$ or $1$.\\
	The only crossings are $t=0$ and $t=1$  since the trace of $\psi(t)$ is $>2$ for $0<t<1.$
	Now, at those points ( $t=0$ and $t=1$)  $\dot\psi_t=
	\left(\begin{smallmatrix}
		\dot a(t)&  \dot d(t)\\
		\dot c(t) & -\dot a(t)+d(t)\dot c(t)
	\end{smallmatrix} \right)$
	so that $S_t= -J_0 \dot \psi_t\psi_t^{-1}=
	\left(\begin{smallmatrix}
		\dot c(t)& -\dot a(t) \\
		-\dot a(t) & \dot a(t)d(t)-\dot d(t)
	\end{smallmatrix} \right).$
 
	Clearly, at the crossings, we have  $ \Ker\psi_t =\R^2$ iff $d(t)=0$ and $ \Ker\psi_t $ is spanned by the first basis element iff $d(t)\neq 0,$ so that  from definition \ref{cross}
	$\Gamma(\psi,t)=\bigl( \dot c(t) \bigr)$
	when $d(t)\neq 0$ and $\Gamma(\psi,t)=
	\left(\begin{smallmatrix}
		0& -\dot a(t) \\
		-\dot a(t) &0
	\end{smallmatrix} \right)$
	when $d(t)= 0.$
	Hence both crossings are regular and $\sign \Gamma(\psi,t)=\sign \dot c(t)$ when $d(t)\neq 0$ and $\sign \Gamma(\psi,t)=0$ when $d(t)=0.$
	Since $d(t)c(t)\ge 0$ for all $t,$ we clearly have $\sign \dot c(0)=\sign d(0)$ and $\sign \dot c(1)= -\sign d(1)$.
	Proposition \ref{murs} then gives $\mu_{\textrm{RS}}(\psi)=\half \Gamma(\psi,0)+\half \sign \Gamma(\psi,1)=\half \sign d(0)-\half\sign d(1).$
 \end{proof}
 
 \begin{remark}
 Robbin and Salamon introduce another index $\mu'_{RS}$ for paths of symplectic matrices built from their index for paths of Lagrangians. Consider the  fixed Lagrangian $L=\{0\}\times\R^n$ in $(\R^{2n},\Omega_0)$, observe that $AL$ is  Lagrangian for any $A\in\Sp(\R^{2n},\Omega_0)$, and define, for  $\psi : [0,1]\rightarrow \Sp(\R^{2n},\Omega_0)$
\begin{equation}
	\mu'_{\textrm{RS}}(\psi) := \mu_{\textrm{RS}}(\psi L,L).
\end{equation}
This index has the following  properties  \cite{RobbinSalamon} :
\begin{itemize}
	\item it is invariant under homotopies with fixed endpoints and two paths with the same endpoints are homotopic with fixed endpoints if and only if they have the same $\mu'_{RS}$ index;
	\item it is additive under catenation of paths;
	\item it has the product property $\mu_{RS}(\psi'\diamond\psi'') = \mu_{RS}(\psi')+\mu_{RS}(\psi'')$;
	\item it vanishes on a path whose image lies in $$ \{ A \in \Sp(\R^{2n},\Omega_0) \, \vert \, \dim AL\cap L = k \}$$
		for a given $k \in \{0, \ldots, n \}$;
	\item $\mu'_{RS}(\psi)=\half \sign B(0) - \half \sign B(1)$ when $\psi_t =
		\left(\begin{smallmatrix}
			\Id & B(t) \\
			0 & \Id
		\end{smallmatrix} \right)$.
\end{itemize}
Robbin and Salamon \cite{RobbinSalamon}  prove   that those properties characterize this index.\\

	The two indices $\mu_{\textrm{RS}}$ and $\mu'_{\textrm{RS}}$ defined on paths of symplectic matrices DO NOT coincide in general.
Indeed, consider the path $\psi : [0,1]\rightarrow \Sp(\R^{2n},\Omega_0) : t\mapsto \psi_t=
	\left(\begin{smallmatrix}
			\Id & 0\\
			C(t) & \Id
	\end{smallmatrix} \right)$.
	Since $\psi_t L \cap L= L \quad \forall t $,	$\mu_{\textrm{RS}2}(\psi) = 0$.
	On the other hand,
	if $\phi =
	\left(\begin{smallmatrix}
		0 & \Id \\
		-\Id & 0
	\end{smallmatrix}\right)$
	and $\psi' = \phi\psi \phi^{-1}$, then $\psi'_t =
	\left(\begin{smallmatrix}
		\Id & -C(t) \\
		0 & \Id
	\end{smallmatrix}\right)$.
	Then
	\begin{equation*}
		\mu'_{\textrm{RS}}(\psi') = \half \sign C(1) - \half \sign C(0)
	\end{equation*}
	which is in general different from $\mu'_{\textrm{RS}}(\psi)$.
Whereas, by \eqref{prop:invariance}, $\mu_{\textrm{RS}}(\psi) = \mu_{\textrm{RS}}(\psi')$.

	The index $\mu'_{\textrm{RS}}$ vanishes on a path whose image lies into one of the $(n+1)$ strata defined by
	$ \{ A \in \Sp(\R^{2n},\Omega_0) \, \vert \, \dim AL\cap L = k \}$ for $0\le k\le n$,
	whereas the index $\mu_{\textrm{RS}}$ vanishes on a path whose image lies into one of the $(2n+1)$ strata defined by the set of symplectic matrices
	whose eigenspace of eigenvalue $1$ has dimension $k$ (for $0\le k\le 2n)$.
	
	However, the two indices $\mu_{\textrm{RS}}$ and $\mu'_{\textrm{RS}}$ coincide on symplectic shears.

 \end{remark}
 
 \section{Characterization of the Robbin-Salamon index}\label{ssection:char}
  In this section, we prove  theorem \ref{thm:Thechar} stated in the introduction.
Before proving this theorem, we show that the Robbin-Salamon index is characterized by the fact that it extends Conley-Zehnder index and has all the properties stated in the previous section. This is made explicit in  Lemma \ref{charRS1}. We then use the characterization of the Conley-Zehnder index given in
Proposition \ref{caract} to give in Lemma \ref{charRS2} a characterization of the Robbin-Salamon index in terms of six properties. 
We use explicitly the normal form of the restriction of a symplectic endomorphism to its generalised eigenspace of eigenvalue $1$ that we have proven in 
\cite{Gutt}  and that we summarize in the following proposition
\begin{proposition}[Normal form for $A_{\vert V_{[\lambda]}}$ for $\lambda=\pm 1.$]\label{normalforms1}
	Let $\lambda=\pm 1$ be an eigenvalue of $A\in Sp(\R^{2n},\Omega_0)$ and let $V_{[\lambda]}$ be the generalized eigenspace of eigenvalue $\lambda$. 
	There exists a symplectic basis of $V_{[\lambda]}$
	in which the matrix associated to the restriction of $A$ to $V_{[\lambda]}$ is a symplectic direct sum of 
	matrices of the form 
	$$
		\left(\begin{array}{cc}
			J(\lambda,r_j)^{-1}&C(r_j,d_j,\lambda)\\
			0& J(\lambda,r_j)^{\tau}
		\end{array}\right)
	$$
	where $C(r_j,d_j,\lambda):=J(\lambda,r_j)^{-1} \operatorname{diag}\bigl(0,\ldots,0, d_j\bigr)$  with $d_j\in \{0,1,-1\}$.
	If $d_j=0$,  then $r_j$ is odd.
	The dimension of the eigenspace of eigenvalue $1$ is given by $2\card \{j \,\vert\, d_j=0\}+\card\{j \,\vert \,d_j\neq0\}$.\\
	For any integer $k\ge 1$, the bilinear form on $\Ker\left( (A-\lambda\id)^{2k}\right)$ defined by
	\begin{eqnarray}
		\hat{Q}_k &:& \Ker\left( (A-\lambda\id)^{2k}\right)\times \Ker\left( (A-\lambda\id)^{2k}\right)\rightarrow \R\nonumber\\
		&&(v,w)\mapsto\Omega((A-\lambda\id)^{k}v,(A-\lambda\id)^{k-1}w)
	\end {eqnarray}
	is symmetric and we have
	\begin{equation}
		\sum_j d_j=\lambda\sum_{k\ge 1}\operatorname{Signature}(\hat{Q}_k)
	\end{equation}
\end{proposition}

\begin{lemma}\label{charRS1}
The  Robbin-Salamon index is characterized by the following properties:
\begin{enumerate}
\item{(Generalization)} it is a correspondence $\mu_{\textrm{RS}}$ which associates a half integer to any continuous path $\psi:[a,b]\rightarrow \Sp(\R^{2n},\Omega_0)$ of symplectic matrices and it coincides with
$\mu_{\textrm{CZ}}$ on paths starting from the identity matrix and ending at a matrix for which $1$ is not an eigenvalue;
\item \emph{(Naturality)} if $\phi,\psi : \left[0,1\right] \rightarrow \Sp(\R^{2n},\Omega_0)$, we have
		 	$\mu_{\textrm{RS}}(\phi\psi\phi^{-1}) = \mu_{\textrm{RS}}(\psi)$;
\item\emph{(Homotopy)} it is invariant under homotopies with fixed end points;
\item\emph{(Catenation)} it is additive under catenation of paths;
\item\emph{(Product)} it has  the product property $\mu_{\textrm{RS}}(\psi'\diamond \psi'')=\mu_{\textrm{RS}}(\psi')+\mu_{\textrm{RS}}(\psi'')$;
\item\emph{(Zero)} it vanishes on any path $\psi:[a,b]\rightarrow\Sp(\R^{2n},\Omega)$ of matrices such that $\dim\Ker (\psi(t)-\Id) =k$ is constant on $ [a,b]$;
\item\emph{(Shear)}on a symplectic shear $,\psi : \left[0,1\right] \rightarrow \Sp(\R^{2n},\Omega_0)$ of the form $$\psi_t =
	\left(\begin{smallmatrix}
		\Id & -tB \\
		0 & \Id
	\end{smallmatrix} \right)=\exp t \left(\begin{smallmatrix}
		0 &- B \\
		0 & 0
	\end{smallmatrix} \right)=\exp t J_0  \left(\begin{smallmatrix}
		0 &0 \\
		0 & B
	\end{smallmatrix} \right)$$
	with $B$ symmetric, it is equal to $
		\mu_{\textrm{RS}}(\psi)= \half \sign B.$
\end{enumerate}
\end{lemma}
\begin{proof}
We have seen in the previous section that the index $\mu_{\textrm{RS}}$ defined by Robbin and Salamon satisfies all the above properties.
To see that those properties characterize this index, it is enough to show (since the group $\Sp(\R^{2n},\Omega_0)$ is connected and since we have the catenation property)
that those properties determine the index of any path starting from the identity. Since it must be a generalization of the Conley-Zehnder index
and must be additive for catenations of paths, it is enough to show that any symplectic matrix $A$ which admits $1$ as an eigenvalue can be linked to a matrix $B$ which does not admit 
$1$ as an eigenvalue by a continuous path whose index is determined by the properties stated. From proposition
 \ref{normalforms1},   there is a basis of $\R^{2n}$ such that $A$ is the  symplectic direct sum of a matrix which does not admit $1$ as eigenvalue and 
matrices of the form 
$$
A^{(1)}_{r_j,d_j}:=
\left(\begin{smallmatrix}
J(1,r_j)^{-1}& J(1,r_j)^{-1} \operatorname{diag}(0,\ldots, 0,d_j)\\
		0&  J(1,r_j)^{\tau}
		\end{smallmatrix}\right);
$$
with $d_j$ equal to $0,1$ or $-1$. The dimension of the eigenspace of eigenvalue $1$  for 
$A^{(1)}_{r_j,d_j}$ is equal to $1$ if $d_j\neq 0$ and is equal to $2$ if $d_j=0$.
 In view of the naturality and the product property of the index, we can consider a symplectic direct sum of paths with the constant path on the symplectic subspace where $1$ is not an eigenvalue and we just have to build a path  in $\Sp(\R^{2r_j},\Omega_0)$ from $A^{(1)}_{r_j,d_j}$
 to a matrix which does not admit $1$ as eigenvalue and whose index is determined by the properties given in the statement.
 This we do by the catenation of three paths : we first build the path $\psi_1: [0,1] \rightarrow \Sp(\R^{2r_j},\Omega_0)$
defined by 
$$ \psi_1(t):=\left(\begin{smallmatrix}
D(t,r_j)^{-1}& D(t,r_j)^{-1}\operatorname{diag}\bigl(c(t),0,\ldots,0,d(t)\bigr)\\
		0& D(t,r_j)^{\tau}
		\end{smallmatrix}\right)
 $$
 with  $D(t,r_j)=\left(\begin{smallmatrix}
 	1&1-t&0&\ldots &\ldots &0\\
	0&e^t&1-t&0&\ldots &0\\
	\vdots &0&\ddots &\ddots &0&\vdots\\
	0&\ldots &0&e^t&1-t &0\\
	0&\ldots &\ldots &0&e^t&1-t\\
	0&\ldots &\ldots &\ldots &0&e^t
\end{smallmatrix}\right)$, \\
and with $c(t)=td_j, ~d(t)=(1-t)d_j$. 
 Observe  that $\psi_1(0)=A^{(1)}_{r_j,d_j}$ and $\psi_1(1)$ is the symplectic direct sum of 
$\left(\begin{smallmatrix}
	1&c(1)=d_j\\
	0& 1
\end{smallmatrix}\right)$ and
$\left(\begin{smallmatrix}
	e^{-1}\Id_{r_j-1}  &0\\
	0&e\Id_{r_j-1}
\end{smallmatrix}\right)$ and this last matrix does not admit $1$ as eigenvalue.

Clearly  $\dim\ker\bigl(\psi_1(t)-\Id\bigr)=2$ for all $t\in [0,1]$ when $d_j=0$; we now prove that $\dim\ker(\psi_1(t)-\Id)=1$ for all $t\in [0,1]$ when $d_j\neq0$. Hence the index of $\psi_1$ must always be zero by the zero property.\\
To prove that $\dim\ker(\psi_1(t)-\Id)=1$ we have to show the non vanishing of  the determinant of the $2r_j-1\times 2r_j-1$ matrix
$$
	{\left(\begin{smallmatrix}
		E^t_{12}&\ldots&\ldots &E^t_{1r_j}&c(t)&0&\ldots&0&E^t_{1r_j}d(t)\\
		e^{-t}-1&E^t_{23}&\ddots &E^t_{2r_j}&0&0&\ldots&0&E^t_{2r_j}d(t)\\
		0 &\ddots&\ddots &\vdots&\vdots&\vdots&&\vdots&\vdots\\
		\vdots&\ddots  &e^{-t}-1&E^t_{r_j-1\, r_j}&0&0&\ldots&0&E^t_{r_j-1\, r_j}d(t)\\
		0&\ldots &0 &e^{-t}-1&0&0&\ldots&0&e^{-t}d(t)\\
		0&\ldots &0 &0&1-t&e^t-1&0&\ddots&0\\
		\vdots&\ldots&\vdots&0&0&1-t&e^t-1&\ddots&0\\
		\vdots&\ldots&\vdots&0&\ldots&0&\ddots&\ddots&0\\[1mm]
		0&\ldots&0  &0&0&\ldots&0&1-t&e^t-1\\
	\end{smallmatrix}\right) }    
 $$
where $E^t:=D(t,r_j)^{-1}$  is upper triangular. 
This determinant is equal to
$$
(-1)^{r_j+1}c(t)(e^{-t}-1)^{r_j-1}(e^{t}-1)^{r_j-1}+(-1)^{r_j-1}d(t)(1-t)^{r_j-1}\det E'(t)
$$
where $E'(t)$ is obtained by deleting the first column and the last line in $E(t)-\Id$ so given by
the $(r_j-1)\times(r_j-1)$ matrix
$$
	{{ \left(\begin{smallmatrix}
		(t-1)e^{-t}&(t-1)^2e^{-2t}&\ldots&\ldots&(t-1)^{r_j-1}e^{-(r_j-1)t}\\[2mm]
		e^{-t}-1&(t-1)e^{-2t}&(t-1)^2e^{-3t}&\ldots&(t-1)^{r_j-2}e^{-(r_j-1)t}\\
		0&e^{-t}-1&(t-1)e^{-2t}&\ddots &(t-1)^{r_j-3}e^{-(r_j-2)t}\\
              \vdots&\ddots\qquad\ddots&&\ddots\quad\quad&\ddots\qquad\vdots\qquad\quad&\\
		\vdots&\qquad\ddots&e^{-t}-1&\quad (t-1)e^{-2t}&(t-1^2)e^{-3t}\\[2mm]
		0&\ldots &0&e^{-t}-1&(t-1)e^{-2t}\\
	\end{smallmatrix}\right) }}.                                                               
$$
Thus  $\det E'(t)=(t-1)(e^{-t}-(e^{-t}-1)) \det F_{r_j-2}(t)$ where
$$
	F_m(t):=  {{ \left(\begin{smallmatrix}
		(t-1)e^{-2t}&(t-1)^2e^{-3t}&\ldots&\ldots&(t-1)^{r_j-2}e^{-(r_j-1)t}\\
		e^{-t}-1&(t-1)e^{-2t}&\ddots &\ldots&(t-1)^{r_j-3}e^{-(r_j-2)t}\\
		0&e^{-t}-1&(t-1)e^{-2t}&\ddots &(t-1)^{r_j-3}e^{-(r_j-2)t}\\
              \vdots&\ddots\qquad\ddots&\qquad\ddots&\quad\quad&\ddots\qquad\vdots\qquad\quad&\\
		\vdots&\qquad\ddots&e^{-t}-1&\quad (t-1)e^{-2t}&(t-1^2)e^{-3t}\\[2mm]
		0&\ldots &0&e^{-t}-1&(t-1)e^{-2t}\\
	\end{smallmatrix}\right) }}.                                                     
$$
and  we have $\det F_m(t)= ((t-1)e^{-2t}-(e^{-t}-1)(t-1)e^{-t})\det F_{m-1}(t)=(t-1)e^{-t}\det F_{m-1}(t)$
so that, by induction on $m$, $\det F_m(t)=(t-1)^me^{-(m+1)t}$ hence the determinant we have to study is\\ 
$
(-1)^{r_j-1}c(t)(2-e^t-e^{-t})^{r_j-1}+d(t)(t-1)^{r_j} \det F_{r_j-2} (t)$
which is equal to \\
$
(-1)^{r_j-1}c(t)(2-e^t-e^{-t})^{r_j-1}+d(t)(t-1)^{r_j}(t-1)^{r_j-2}e^{-(r_j-1)t}$ 
hence to 
$$
	c(t)(e^t+e^{-t}-2)^{r_j-1}+d(t)(1-t)^{2r_j-2}e^{-(r_j-1)t}
$$
 which never vanishes if $c(t)=td_j$ and
$d(t)=(1-t)d_j$ since $e^t+e^{-t}-2$ and $(1-t)$ are $\ge 0$.

We then  construct a path $\psi_2: [0,1] \rightarrow \Sp(\R^{2r_j},\Omega_0)$ which is constant on the symplectic subspace where $1$ is not an eigenvalue  and which
 is a symplectic shear on the first two dimensional symplectic vector space, i.e. 
 $$
 	\psi_2(t):=
	\left(
		\begin{smallmatrix}
			1&(1-t)d_j\\
			0& 1
		\end{smallmatrix}
	\right)
	\diamond
	\left(
		\begin{smallmatrix}
			e^{-1}\Id_{r_j-1}  &0\\
			0&e\Id_{r_j-1}
		\end{smallmatrix}
	\right);                                           
 $$  
then the index of $\psi_2$ is equal to $\half\sign d_j$. 
Observe that $\psi_2$ is constant if $d_j=0$; then the index of $\psi_2$ is zero. In all cases $\psi_2(1)=\Id_2\diamond  \left(\begin{smallmatrix}e^{-1}\Id_{r_j-1}  &0\\0&e\Id_{r_j-1}\end{smallmatrix}\right)$. \\
We then build $\psi_3: [0,1] \rightarrow \Sp(\R^{2r_j},\Omega_0)$ given by
$$
	\psi_3(t):=
	\left(
		\begin{smallmatrix}
			e^{-t}& 0 \\ 
			0 & e^{t}
                  \end{smallmatrix}
	\right)
	\diamond    
	\left(
		\begin{smallmatrix}
			e^{-1}\Id_{r_j-1}  &0\\
			0&e\Id_{r_j-1}  
		\end{smallmatrix}
	\right)
$$
which is the direct sum of a path whose Conley-Zehnder index is known and a constant path whose index is zero.
Clearly $1$ is not an eigenvalue of $\psi_3(1)$.
\end{proof}

Combining the above with the characterization of the Conley-Zehnder index, we now prove:
\begin{lemma}\label{charRS2}
	The Robbin-Salamon index for a path of symplectic matrices is characterized by the following properties:
	\begin{itemize}
		\item\emph{(Homotopy)} it is invariant under homotopies with fixed end points;
		\item\emph{(Catenation)} it is additive under catenation of paths;
		\item\emph{(Zero)} it vanishes on any path $\psi:[a,b]\rightarrow\Sp(\R^{2n},\Omega)$ of matrices such that $\dim\Ker (\psi(t)-\Id) =k$ is constant on $ [a,b]$;
		\item\emph{(Product)} it has  the product property $\mu_{\textrm{RS}}(\psi'\diamond \psi'')=\mu_{\textrm{RS}}(\psi')+\mu_{\textrm{RS}}(\psi'')$;
		\item\emph{(Signature)} if $S=S\tr \in \R^{2n\times 2n}$ is a symmetric non degenerate matrix with all eigenvalues of absolute value
			$<2\pi $ and if $\psi(t) = \operatorname{exp}(J_0St)$ for $t\in \left[0,1\right],$ then
			$\mu_{\textrm{RS}}(\psi) = \half \sign S$ where $\sign S$ is the signature of $S$;
		\item\emph{(Shear)} if $\psi_t = \exp t J_0  \left(\begin{smallmatrix}
		0 &0 \\
		0 & B
	\end{smallmatrix} \right)$ for $t\in \left[0,1\right],$
	with $B$ symmetric, then $
		\mu_{\textrm{RS}}(\psi)= \half \sign B.$
	\end{itemize}
\end{lemma}

\begin{proof}
Remark first that the invariance by homotopies with fixed end points, the additivity under catenation and the zero property
imply the naturality; they also imply the constancy on the components of $ \SP(n).$ 
The signature property stated above is the signature property which arose in the characterization of the Conley-Zehnder index given in proposition \ref{caract}.
To be sure that our index is a generalization of the Conley-Zehnder index, there remains just  to prove the loop property.
Since the product of a loop $\phi$ and a path $\psi$ starting at the identity is homotopic to the catenation of $\phi$ and $\psi$,
it is enough to prove that the index of a loop $\phi$  with $\phi(0) = \phi(1) = \Id$ is given by
$ 2 \operatorname{deg}(\rho \circ \phi)$.
Since two loops $\phi$ and $\phi'$ are homotopic if and only if $\deg(\rho\circ\phi)=\deg(\rho\circ\phi'),$ it is enough to consider the loops $\phi_n$ defined by $\phi_n(t):=
		\left(\begin{smallmatrix} 
			\cos 2\pi nt & -\sin 2\pi nt \\
			\sin 2\pi nt & \cos 2\pi nt \end{smallmatrix}\right)\diamond\Id $;
		  since $\phi_n(t)=\bigl(\phi_1(t)\bigr)^n$, it is enough to show, using the homotopy, catenation, product and zero properties that 	the index of the loop given by $\phi(t)=	\left(\begin{smallmatrix} 
			\cos 2\pi t & -\sin 2\pi t \\
			\sin 2\pi t & \cos 2\pi t \end{smallmatrix}\right)$ for $t\in [0,1]$ is equal to $2$. This is true, using the signature property,			writing $\phi$ as the catenation of the path $\psi_1(t):=\phi(\frac{t}{2})=\exp tJ_0\left(\begin{smallmatrix} 
			\pi  & 0\\
			0 & \pi  \end{smallmatrix}\right)$ for $t\in [0,1]$  whose index is $1$ and the path $\psi_2(t):=\phi(\frac{t}{2})=\exp tJ_0\left(\begin{smallmatrix} 
			\pi  & 0\\
			0 & \pi  \end{smallmatrix}\right)$ for $t\in [1,2]$. 
			We introduce the  path in the reverse direction  $\psi^-_2(t):=\exp -tJ_0\left(\begin{smallmatrix} 
			\pi  & 0\\
			0 & \pi  \end{smallmatrix}\right)$ for  $t\in [0,1]$ whose index is $-1$; since the catenation of $\psi^-_2$ and $\psi_2$ is homotopic to
			      the constant path whose index is zero,  the index of  $\phi_1$ is given by the index of $\psi_1$ minus the index of $\psi^-_2$
			      hence is equal to $2$.
\end{proof}
We are now ready to prove the characterization of the Robbin-Salamon index stated in the introduction.\\

\begin{proof}[Proof of theorem \ref{thm:Thechar}] 
Observe that any symmetric matrix can be written as the symplectic direct sum of a non degenerate symmetric matrix $S$ and a matrix $S'$ of the form
$ \left(\begin{smallmatrix}
		0 &0 \\
		0 & B
	\end{smallmatrix} \right)$ where $B$ is symmetric and may be degenerate. The index of the path $\psi_t=\exp t J_0S'$ is equal to the index of the path $\psi'_t=\exp t \lambda J_0S'$
	for any $\lambda>0$. Hence the signature and shear conditions, in view of the product condition, can be simultaneously written as: 
	if $S=S\tr \in \R^{2n\times 2n}$ is a symmetric  matrix with all eigenvalues of absolute value
			$<2\pi $ and if $\psi(t) = \operatorname{exp}(J_0St)$ for $t\in \left[0,1\right],$ then
			$\mu_{\textrm{RS}}(\psi) = \half \sign S$. This is the normalization condition stated in the theorem.
			
From Lemma \ref{charRS2}, we just have to prove that the product property is a consequence of the other properties. We prove it for paths with values 
in $\Sp(\R^{2n},\Omega_0)$ by induction on $n$, the case $n=1$ being obvious.
Since $\psi'\diamond\psi''$ is homotopic with fixed endpoints to the catenation 
of $\psi'\diamond\bigl(\psi''(0)\bigr)$ and $\bigl(\psi'(1)\bigr)\diamond\psi''$, it is enough to show that the
index of $A\diamond\psi$ is equal to the index of $\psi$ for any fixed $A\in \Sp(\R^{2n'},\Omega_0)$ with $n'<n$ and any
continuous path $ \psi : [0,1] \rightarrow  \Sp(\R^{2n''},\Omega_0)$ with $n''<n$.

Using the proof of lemma  \ref{charRS1},  any symplectic matrix $A$ can be linked by a path $\phi(s)$
with constant dimension of the $1$-eigenspace to a matrix of the form $\operatorname{exp}(J_0S')$ with $S'$ a 
symmetric  $n'\times n'$ matrix with all eigenvalues of absolute value $<2\pi $. The index of $A\diamond\psi$
is equal to the index of $\operatorname{exp}(J_0S')\diamond \psi$; indeed  $A\diamond\psi$ is homotopic with fixed endpoints
to the catenation of the three paths $\phi_s\diamond\psi(0)$, $\operatorname{exp}(J_0S')\diamond \psi$ and the path $\phi_s\diamond\psi (1)$
in  the reverse order, and the index of the first and third paths are zero since the dimension of the $1$-eigenspace does not vary along those paths.

Hence it is enough to show that the index of $\operatorname{exp}(J_0S')\diamond \psi$ is the same as the index of $\psi$. This is true because  the map 
$\mu$
sending  a path $\psi$  in $\Sp(\R^{2n''},\Omega_0)$ (with $n''<n$) to the index of $\operatorname{exp}(J_0S')\diamond \psi$ has the four properties stated in the theorem,
and these characterize the Robbin-Salamon index for those paths by induction hypothesis. 
It is clear that $\mu$ is invariant under homotopies, additive for catenation and equal to
zero on paths $\psi$ for which the dimension of the $1$-eigenspace is constant.  Furthermore
$\mu(\exp t(J_0S))$ which is the index of $\exp(J_0S')\diamond \exp t(J_0S)$ is equal to $\half \sign S$, because the path
$\exp tJ_0(S'\diamond S)$ whose index is $\half \sign (S'\diamond S)= \half \sign S' + \half \sign S$ is homotopic with fixed 
endpoints with the catenation of $\exp t(J_0S')\diamond \Id=\exp tJ_0(S'\diamond 0)$, whose index is $\half \sign S'$,
and the path $\exp(J_0S')\diamond \exp t(J_0S)$.
\end{proof}			
	
\section{A formula for the Robbin-Salamon index}\label{ssection:formula}

Let $\psi : [0,1]\rightarrow\Sp(\R^{2n},\Omega_{0})$ be a path of symplectic matrices.
The symplectic transformation $\psi(1)$ of $V=\R^{2n}$ decomposes as
$$
	\psi(1) = \psi^{\star}(1)\diamond \psi^{(1)}(1)
$$
where $ \psi^{\star}(1)$  does not admit $1$ as eigenvalue and
$\psi^{(1)}(1)$ is the restriction of $\psi(1)$ to the generalized eigenspace of eigenvalue $1$  
$$
	\left.\psi(1)\right\vert_{V_{[1]}}.
$$
By proposition
 \ref{normalforms1}, there exists a
symplectic matrix $A$ such that $A\psi^{(1)}(1) A^{-1}$ is equal to
\begin{eqnarray}
	&&\psi^{\star}(1)\diamond
		\left(
			\begin{smallmatrix}
				J(1,r_1)^{-1}&C\bigl(r_1,d_1^{(1)},1\bigr)\\
				0& J(1,r_1)^{\tau}
			\end{smallmatrix}
		\right)
		\diamond\cdots\diamond
		\left(
			\begin{smallmatrix}
				J(1,r_k)^{-1}&C\bigl(r_k,d_k^{(1)},1\bigr)\\
				0& J(1,r_k)^{\tau}
			\end{smallmatrix}
		\right)
		\qquad\qquad~ \label{int}\\
	&&\qquad\qquad\qquad\qquad \diamond
		\left(
			\begin{smallmatrix}
				J(1,s_1)^{-1}&0\\
				0& J(1,s_1)^{\tau}
			\end{smallmatrix}
		\right)
		\diamond\cdots\diamond
		\left(
			\begin{smallmatrix}
				J(1,s_l)^{-1}&0\\
				0& J(1,s_l)^{\tau}
			\end{smallmatrix}
		\right)
		\nonumber
\end{eqnarray}
with  each $d^{(1)}_j=\pm 1$.
Since $\Sp(\R^{2n},\Omega_{0})$ is connected, there is a path $\varphi:[0,1] \rightarrow\Sp(\R^{2n},\Omega_{0})$ such that $\varphi(0)=\id$ and $\varphi(1)=A$.
We define
$$
	\psi_{\uppercase\expandafter{\romannumeral 1}} : [0,1] \rightarrow\Sp(\R^{2n},\Omega_{0}): t\mapsto \varphi(t)\psi(t)\bigl(\varphi(t)\bigr)^{-1}.
$$
It is a  path from $\psi(1)$ to the matrix defined in \ref{int}. Clearly, $\mu_{RS}(\psi_{\uppercase\expandafter{\romannumeral 1}}) = 0$ and $\rho$ is constant
on $\psi_{\uppercase\expandafter{\romannumeral 1}}$.

Let $\psi_{\uppercase\expandafter{\romannumeral 2}} : [0,1] \rightarrow\Sp(\R^{2n},\Omega_{0})$ be the path from
$\psi_{\uppercase\expandafter{\romannumeral 1}}(1)$ to 
$$
	\psi^{\star}(1)\diamond
	\left(
		\begin{smallmatrix}
			1&d^{(1)}_1\\
			0&1
		\end{smallmatrix}
	\right)
	\diamond\cdots\diamond
	\left(\begin{smallmatrix}
		1&d^{(1)}_k\\
		0&1
	\end{smallmatrix}\right)\diamond
		\left(\begin{smallmatrix}
		1&0\\
		0&1
	\end{smallmatrix}\right)
	\diamond\cdots\diamond
	\left(\begin{smallmatrix}
		1&0\\
		0&1
	\end{smallmatrix}\right)\diamond
	\left(\begin{smallmatrix}
		e^{-1}\id&0\\
		0& e\id
	\end{smallmatrix}\right)
$$
defined as in the proof of lemma \ref{charRS1} in each block by
$$
	\left(
		\begin{smallmatrix}
			D(t,r_j)^{-1}& D(t,r_j)^{-1}\operatorname{diag}\bigl(td^{(1)}_j,0,\ldots,0,(1-t)d^{(1)}_j\bigr)\\
			0& D(t,r_j)^{\tau}
		\end{smallmatrix}
	\right)
 $$
 with  $D(t,r_j)=
 \left(
 	\begin{smallmatrix}
 		1&1-t&0&\ldots &\ldots &0\\
		0&e^t&1-t&0&\ldots &0\\
		\vdots &0&\ddots &\ddots &0&\vdots\\
		0&\ldots &0&e^t&1-t &0\\
		0&\ldots &\ldots &0&e^t&1-t\\
		0&\ldots &\ldots &\ldots &0&e^t
	\end{smallmatrix}
\right)$. 
Note that  $\mu_{RS}(\psi_{\uppercase\expandafter{\romannumeral 2}}) = 0$ since the eigenspace of eigenvalue $1$ has constant dimension
and $\rho$ is constant on $\psi_{\uppercase\expandafter{\romannumeral 2}}$.\\
We define $\psi_{\uppercase\expandafter{\romannumeral 3}} : [0,1] \rightarrow\Sp(\R^{2n},\Omega_{0})$
from $\psi_{\uppercase\expandafter{\romannumeral 2}}(1)$ to 
$$
	\psi^{\star}(1)\diamond
	\left(
		\begin{smallmatrix}
			\Id&0\\
			0& \Id
		\end{smallmatrix}
	\right)
	\diamond
	\left(
		\begin{smallmatrix}
			e^{-1}\Id&0\\
			0& e\Id
		\end{smallmatrix}
	\right)
$$
which is given on each block $
\left(
	\begin{smallmatrix}
		1&d^{(1)}_j\\
		0&1
	\end{smallmatrix}
\right)
$ by $
\left(
	\begin{smallmatrix}
		1&(1-t)d^{(1)}_j\\
		0&1
	\end{smallmatrix}
\right)$.
Note that  $\mu_{RS}(\psi_{\uppercase\expandafter{\romannumeral 3}}) = \half\sum_j d^{(1)}_j$ 
by proposition \ref{RSforshears}
and $\rho$ is constant
on $\psi_{\uppercase\expandafter{\romannumeral 3}}$.\\	
Finally, consider $\psi_{\uppercase\expandafter{\romannumeral 4}} : [0,1] \rightarrow\Sp(\R^{2n},\Omega_{0})$
from $\psi_{\uppercase\expandafter{\romannumeral 3}}(1)$ to 
$$
	\psi^{\star}(1)\diamond
	\left(
		\begin{smallmatrix}
			e^{-1}\Id&0\\
			0& e\Id
		\end{smallmatrix}
	\right)
$$
which is given by $\psi^{\star}(1)\diamond
\left(
	\begin{smallmatrix}
		e^{-t}&0\\
		0&e^{t}
	\end{smallmatrix}
\right)
\diamond
\left(
	\begin{smallmatrix}
		e^{-1}\Id&0\\
		0& e\Id
	\end{smallmatrix}
\right)$.
Note that  $\mu_{RS}(\psi_{\uppercase\expandafter{\romannumeral 4}}) =0$, 
$\rho$ is constant
on $\psi_{\uppercase\expandafter{\romannumeral 4}}$ and	
 $\psi_{\uppercase\expandafter{\romannumeral 4}}(1)$ is in 	
$\Sp^{\star}(\R^{2n},\Omega_{0})$.
Since two paths of matrices with fixed ends are homotopic if and only if their image under $\rho$
are homotopic,  the catenation of the paths $\psi_{\uppercase\expandafter{\romannumeral  3}}$ and
$\psi_{\uppercase\expandafter{\romannumeral 4}}$  is homotopic to any path from $ \psi_{\uppercase\expandafter{\romannumeral  1}}(1)$
 to $\psi^{\star}(1)\diamond
\left(
	\begin{smallmatrix}
		e\Id&0\\
		0& e^{-1}\Id
	\end{smallmatrix}
\right)$
of the form $\psi^{\star}(1) \diamond ) \diamond \phi_1(t)$ where $\phi_1(t)$  has only real positive eigenvalues.
We proceed similarly  for $\psi(0)$ and we get
\begin{theorem}\label{RSexplicite}
	Let $\psi : [0,1]\rightarrow\Sp(\R^{2n},\Omega_{0})$ be a path of symplectic matrices.
	Decompose $\psi(0) = \psi^{\star}(0)\diamond \psi^{(1)}(0)$ and 
	$\psi(1) = \psi^{\star}(1)\diamond \psi^{(1)}(1)$
	where $ \psi^{\star}(0)$ (resp. $ \psi^{\star}(1)$)  does not admit $1$ as eigenvalue and
	$\psi^{(1)}(0)$ (resp. $\psi^{(1)}(1)$) is the restriction of $\psi(0)$ (resp. $\psi(1)$) to the generalized eigenspace of eigenvalue $1$
	of $\psi(0)$ (resp. $\psi(1)$).  
	Consider a prolongation  $\Psi:[-1,2]\rightarrow \Sp(\R^{2n},\Omega_{0})$ of $\psi$ such that
	\begin{itemize}
		\item  $\Psi(t)=\psi(t) \,\forall t\in [0,1]$;
		\item $\Psi\bigl(-\half\bigr) =\psi^{\star}(0)\diamond
			\left(
				\begin{smallmatrix}
					e^{-1}\Id&0\\
					0& e\Id
				\end{smallmatrix}
			\right)$ and $\Psi(t)=\psi^{\star}(0) \diamond \phi_0(t)$ where $\phi_0(t)$  has only real positive eigenvalues
			 for $t\in \bigl[-\half,0\bigr]$;
		\item $\Psi\bigl(\frac{3}{2}\bigr)=\psi^{\star}(1)\diamond
			\left(
				\begin{smallmatrix}
					e^{-1}\Id&0\\
					0& e\Id
				\end{smallmatrix}
			\right)$
			and $\Psi(t)=\psi^{\star}(1) \diamond \phi_1(t)$ where $\phi_1(t)$  has only real positive eigenvalues
			 for  $t\in \bigl[1,\frac{3}{2}\bigr]$;
		\item $\Psi(-1)=W^{\pm}$, $\Psi(2)=W^{\pm}$ and  $\Psi(t)\in \Sp^{\star}(\R^{2n},\Omega_{0})$ for $t\in \bigl[-1,-\half\bigr]\cup\bigl[\frac{3}{2},2\bigr]$.
	\end{itemize}
	Then
		$$
		\mu_{RS}(\psi) =\deg(\rho^2\circ \Psi) + \half\sum d_{i}^{(0)} - \half\sum d_{j}^{(1)}.
	$$
\end{theorem} 
Remark that we can replace in the formula above $\rho$ by $\tilde{\rho}$ as in proposition \ref{compdef}.\\
By proposition  \ref{normalforms1}, we have theorem \ref{RSexpl} :
	$$
		\mu_{RS}(\psi) = \deg(\rho^{2}\circ\Psi)+ \half \sum_{k=1}^{dim V}\sign\Bigl(\hat{Q}_k^{(\psi(0))}\Bigr)
		 -\half \sum_{k=1}^{dim V}\sign\Bigl(\hat{Q}_k^{(\psi(1))}\Bigr).
	$$

\begin{remark}
	The advantage of this new formula is that to compute the index of a path whose crossing with the Maslov cycle is non transverse we do not need to perturb the path.
	The drawback is that we have to extend the initial path.
\end{remark}

\bibliographystyle{alpha}
\nocite{*}
\bibliography{FinalgeneralizedCZ.bib}

\end{document}